\documentclass[10pt]{article}
\usepackage{times,amsmath,amssymb,exscale,mathrsfs}
\usepackage{pst-all}
\usepackage{graphicx}
\usepackage{xcolor} % For colored text
\usepackage{epsfig}
\newtheorem{definition}{Definition}%[section]
\newtheorem{theorem}{Theorem}%[section]
\newtheorem{remark}{Remark}%[section]
\newtheorem{example}{Example}%[section]
\newtheorem{lemma}{Lemma}
\newenvironment{proof}{\noindent{\bf Proof:}}{$\Box$\medskip}
\newcommand{\C}{\mathbb{C}}
\newcommand{\calC}{\mathcal{C}}
\newcommand{\g}{\mathfrak{g}}
\newcommand{\G}{\mathcal{G}}
\newcommand{\Ham}{\mathcal{H}}
\newcommand{\I}{\mathcal{I}}
\newcommand{\bk}{{\bf k}}
\newcommand{\bl}{{\bf l}}
\newcommand{\R}{\mathbb{R}}
\newcommand{\calR}{\mathcal{R}}
\newcommand{\W}{\mathcal{W}}
\newcommand{\Z}{\mathbb{Z}}
\newcommand{\uno}{{\, 1\!\! 1\,}}
\newcommand{\sh}{{\,\scriptstyle \sqcup\!\sqcup\,}}

\title{Word series for dynamical systems and their numerical integrators}
\author{A. Murua\footnote{Konputazio Zientziak eta A.\ A.\  Saila, Informatika
 Fakultatea, UPV/EHU, E--20018 Donostia--San Sebasti\'{a}n,  Spain. Email: Ander.Murua@ehu.es}
\  and J.M. Sanz-Serna\footnote{(Corresponding author) Departamento de  Matem\'aticas, Universidad Carlos III de Madrid, Avenida de la Universidad 30, E--28911 Legan\'es (Madrid), Spain.
 Email: jmsanzserna@gmail.com}}
\date{\today}

\begin{document}
\maketitle
\begin{abstract}We study word series and extended word series, classes of formal series for the analysis of some dynamical systems and their discretizations. These series are similar to but more compact than B-series. They may be composed among themselves by means of a simple rule.
While word series have appeared before in the literature,  extended word series are introduced in this paper. We exemplify the use of extended word series by studying the reduction to normal form and averaging of some perturbed integrable problems. We also provide a detailed analysis of the behaviour of splitting numerical methods for those problems.
\end{abstract}

\medskip

\noindent{\bf Keywords and sentences:} Word series, extended word series, B-series, words, Hopf algebras, shuffle algebra, Lie groups, Lie algebras, Hamiltonian problems, integrable problems, normal forms, averaging, splitting algorithms, processing numerical methods, modified systems, resonances.

\medskip


\medskip

\noindent Communicated by Christian Lubich
\section{Introduction}

In this paper we study word series and extended word series, classes of formal series of functions for the analysis of some dynamical systems and their discretizations. We exemplify the use of extended word series by studying the reduction to normal form of some perturbed integrable problems. We also provide a detailed analysis of the behaviour of splitting numerical methods for those problems. Word series are patterned after B-series
\cite{HW}, a commonly used tool in the study of numerical integrators; while B-series are parametrized by rooted trees, word series are parametrized by words built from the letters of an alphabet.
Series of {\em differential operators}  parametrized by words (including the Chen-Fliess series) are very
common in control theory \cite{nuevo} and dynamical systems \cite{fm} and have also been used in numerical analysis (see, among others, \cite{ander}, \cite{anderfocm}, \cite{k}). Word series are mathematically equivalent to  series of differential operators, but being series of {\em functions}, they are handled in a way very similar to the way B-series are used by numerical analyst.
Word series, as defined here, have appeared before in the literature, explicitly \cite{orlando}, \cite{part3} or implicitly \cite{part2}.  Extended word series are introduced in this paper.

B-series, introduced by Hairer and Wanner in 1974 \cite{HW}, provided the first example of the application of formal series of functions to the theory of numerical integrators (see \cite{china} for a historical survey). B-series, particularly adapted to Runge-Kutta and related methods, give a convenient, systematic way of performing the nontrivial algebraic manipulations needed to write the expansion of the local error in powers of the stepsize. In addition, they facilitate the construction of integrators found by composing simpler integrators; such a construction is required e.g. when investigating the effective order of Runge-Kutta methods \cite{butcher69}, \cite{bss}. The usefulness of B-series stems from the fact that the composition of two B-series is again a B-series whose coefficients may be written down explicitly and are universal in the sense that they are independent of the particular differential system being integrated. B-series and their extensions grew more important within the notion of  geometric integration \cite{gi}. In 1994 Calvo and one of the present authors \cite{canonical}  showed  how the conditions for a Runge-Kutta scheme to be symplectic may be advantageously derived by examining the corresponding B-series.  Hairer's article \cite{hairer} started the use of B-series to find explicitly modified systems. Since those pioneering contributions the role of B-series and its generalizations \cite{ander} in geometric integration has kept growing as it may be seen in the treatise \cite{hlw}.

The word series and extended word series considered here are, when applicable, more convenient than B-series. A reason for this convenience is that they are more compact; in fact the coefficients of a B-series are parametrized by (possibly coloured or decorated) rooted trees and there are many more
rooted trees with $n$ vertices than words with $n$ letters. A second advantage of word series and extended word series over B-series is that for words the composition rule (see (\ref{eq:convol}) and (\ref{eq:act})) is much simpler than for rooted trees.

An overview of the contributions of this paper is as follows.

Section \ref{sec:words}  gives a summary of the rules to manipulate word series.
A group $\G$ is introduced that plays the role played by the Butcher group in the theory of B-series. The solution of the differential system being integrated and some numerical methods, including splitting algorithms, may be represented by elements of $\G$. We also identify the Lie algebra $\g$ associated with $\G$ and the corresponding bracket. This material  is very much related to the theory of Hopf algebras \cite{anderfocm}, \cite{brouder}; however   Section \ref{sec:words}  has been written with an audience of computational scientist in mind and a number of more algebraic considerations have been postponed  to Section~\ref{s:technical}.

Extended word series are introduced in Section~\ref{sec:extended} to cope with perturbed integrable problems; roughly speaking we treat problems that may be seen as arbitrary perturbations of systems that, in suitable variables, may be cast in the form
$(d/dt)y = 0$, $(d/dt)\theta = \omega$ (in the language of classical mechanics \cite{arnoldmec} $y$/$\theta$ would correspond to action/angle variables). We describe the relevant group $\overline{\G}$ and algebra $\overline{\g}$.

In Section~\ref{s:normal} we show how to use extended word series to bring perturbed integrable problems to {\em normal form}, i.e. how to change variables to reduce the system being analyzed to a form as simple as possible. As distinct from standard ways of finding normal forms, the extended word series approach does not rely on the vector field being polynomial. Furthermore,  the computations required  here are {\em universal} (in the sense of \cite{orlando}): they are independent of the particular system under consideration.

For highly oscillatory problems the reduction to normal form is very much related to the process of {\em averaging} out the oscillatory components of the solution and therefore Section~\ref{s:normal} extends the material in the series of papers \cite{part1}, \cite{part2}, \cite{orlando},  \cite{part3}. Furthermore normal forms readily lead to the explicit computation of (formal) invariant quantities of dynamical systems and their discretizations, an issue not covered in this paper and treated in the follow-up article
\cite{jl}.

Section~\ref{sec:split} is devoted to the study of general splitting algorithms to simulate perturbed integrable problems. We show how our algebraic approach leads to a convenient expansion of the local error. It is well known that  the behaviour of the corresponding global error as $h$ varies is unfortunately extremely complex,  as it depends on arithmetic relations between $h$ and the periods present in the solution. Extended word series provide a powerful instrument to analyze that behaviour. In fact, two different approaches are put forward here. In the first, the integrator is processed, i.e.\  subjected to  changes of variables,  to remove oscillatory components. The second approach is based on constructing a modified system for the integrator and then bringing the modified system to normal form. Of much interest is the fact that the validity of the modified system holds even if $h$ is not small relative to the periods in the problem (cf. the use by Hairer and Lubich of modulated Fourier expansions \cite{hlw}). The techniques  in this section may be readily applied to the  construction and analysis of improved integrators, such as those considered in e.g.\ \cite{molly1}, \cite{molly2}; this will be the subject of future work.

Section~\ref{s:technical} contains proofs and technical material and there is an Appendix devoted to the practical applicability of splitting integrators.

\section{Word series}
\label{sec:words}
This section presents word series and provides a summary of the rules for their application. The presentation has computational scientist in mind and focuses on essential features;  additional details and proofs are given in Section~\ref{ss:technicalwords}, where the approach is more algebraic. Until further notice all functions are assumed to be smooth.

\subsection{Definition of word series}
We consider the \(D\)-dimensional initial-value problem given by
\begin{equation}\label{eq:ic}
x(0) = x_0
\end{equation}
and
\begin{equation}\label{eq:ode}
\frac{d}{dt} x= \sum_{a\in A} \lambda_a(t) f_a(x),
\end{equation}
where $t$ is the (real) independent variable, $A$ is a  finite or infinite countable set of indices and for each $a\in A$, $\lambda_a$ is a scalar-valued function  and $f_a$  a  $D$-vector-valued map.

It is well known that the solutions of (\ref{eq:ode}) may be expanded formally as follows. Associated with each vector field $f_a$ in (\ref{eq:ode}), there is a first-order linear differential operator $E_{a}$: if $g$ is a scalar-valued  function, then the function $E_ag$ is defined by
\begin{equation}\label{eq:revision}
E_a g(x) = \sum_{j=1}^D f^j_a(x)\frac{\partial}{\partial x^j} g(x)
\end{equation}
(superscripts denote components of vectors). We shall also let $E_a$ act on vector-valued mappings; it is then understood that the operator is applied  componentwise. If $x(t)$ satisfies (\ref{eq:ode}), the chain rule yields
\begin{equation*}
\frac{d}{dt} g(x(t)) = \sum_{a\in A} \lambda_a(t) (E_ag)(x(t))
\end{equation*}
or
\begin{equation*}
g(x(t)) = g(x(0)) + \sum_{a\in A} \int_0^t dt_1\, \lambda(t_1) (E_ag)(x(t_1)).
\end{equation*}
The same procedure may be now applied with $(E_ag)(x(t_1))$ in lieu of $g(x(t))$ to rewrite the last equation as
\begin{eqnarray*}
g(x(t)) &=& g(x(0)) + \sum_{a\in A}\int_0^t dt_1\, \lambda(t_1) (E_ag)(x(0))\\&&+ \sum_{a\in A} \sum_{b\in A}\int_0^t dt_1\, \lambda(t_1) \int_0^{t_1}dt_2\,\lambda(t_2)(E_b(E_ag))(x(t_2)).
\end{eqnarray*}
By continuing this Picard iteration and setting $g$ equal to the identity function $g(x)=x$, we find that the solution of (\ref{eq:ic})--(\ref{eq:ode}) has the formal expansion
\begin{equation}\label{eq:expan}
 x(t) = x_0 + \sum_{n=1}^\infty \sum_{a_1,\dots,a_n\in A}\alpha_{a_1\cdots a_n}(t) f_{a_1\cdots a_n}\!(x_0),
\end{equation}
where   the vector-valued mappings
$ f_{a_1\cdots a_n}\!(x)$ and the scalar-valued  functions $\alpha_{a_1\cdots a_n}$ satisfy the recursions
\begin{equation}\label{eq:wbf}
f_{a_1\cdots a_n}\!(x) = \partial_x f_{a_2\cdots a_n}\!(x)\,f_{a_1}\!(x), \quad n >1,
\end{equation}
($\partial_x f_{a_2\cdots a_n}\!(x)$ denotes the value at $x$ of the Jacobian matrix of $f_{a_2\cdots a_n}$)
and
\begin{equation}\label{eq:alfa1}
\alpha_{a_1}\!(t) = \int_0^t \lambda_{a_1}\!(t_1)\,dt_1,
\end{equation}
$$
\alpha_{a_1\cdots a_n}\!(t) = \int_0^t \lambda_{a_n}\!(t_n)\,\alpha_{a_1\cdots a_{n-1}}\!(t_n)\,dt_n, \quad n >1.
$$
For future reference we note that
$$
\alpha_{a_1\cdots a_n}\!(t) = \int_0^t dt_n\,\lambda_{a_n}\!(t_n)\int_0^{t_n}dt_{n-1}\,\lambda_{a_{n-1}}\!(t_{n-1})\cdots \int_0^{t_2} dt_1\,\lambda_{a_1}\!(t_1),
$$
or
\begin{equation}\label{eq:alpha}
\alpha_{a_1\cdots a_n}\!(t) =\int\cdots \int_{{\mathcal S}_n(t)} \lambda_{a_1}\!(t_1)\cdots \lambda_{a_n}\!(t_n)\, dt_1\cdots dt_n,
\end{equation}
where the $n$-fold integral is taken over the simplex
$$
{\mathcal S}_n(t) =
\{ (t_1,\dots,t_n) \in \R^n : 0\leq t_1\leq \cdots \leq t_n \leq t \}.
$$

Let us present some examples (more may be  seen in \cite{anderfocm}):
\begin{enumerate}
\item In the simplest illustration, the set $A$ has only one element $a$ and the corresponding $\lambda$ takes the value $1$ for each $t$. Then (\ref{eq:ode}) is the  {\em autonomous} system $(d/dt) x = f_a(x)$. For each $n$, the inner sum in (\ref{eq:expan}) comprises a single term and from (\ref{eq:alpha}) the corresponding coefficient is found  to be
     $\alpha_{a\cdots a}\!(t)=t^n/n!$. In this case (\ref{eq:expan}) is the standard Taylor expansion of $x(t)$.
\item The {\em autonomous} system $(d/dt) x = F(x)$ with the right-hand side split as $F(x)$ $= f_a(x)+f_b(x)$ is of the form (\ref{eq:ode}) with $A =\{a,b\}$ and
$\lambda_a(t) = \lambda_b(t)= 1$. For each $n$, the inner sum in (\ref{eq:expan}) comprises $2^n$ terms and each of them has a coefficient $t^n/n!$. The expansion (\ref{eq:expan}) is  the Taylor series for $x(t)$ written in terms of the pieces $f_a$ and $f_b$ rather than in terms of $F$, a format that is useful in the analysis of splitting numerical integrators. It is of course possible to split $F$ in $m>2$ parts  or even in infinitely many parts; then $A$ has $m$ or infinitely many elements. In all cases the integral (\ref{eq:alpha}) has the value $t^n/n!$.
\item Let $\omega>0$ be a fixed number. If $A =\Z$ (the set of all integers)
    and, for $k \in \Z$,
    \(
    \lambda_k\!(t) = \exp(ik\omega t),
    \)
    then   (\ref{eq:ode}) is a {\em non-autonomous} system  $2\pi/\omega$-periodic in $t$ with the right-hand side expanded in Fourier series. Here the functions $\lambda_k$ possess complex values; if the system (\ref{eq:ode}) is real then the mappings $f_k$ take values in $\C^D$ and $f_k$ is the complex conjugate of $f_{-k}$ for each $k\in \Z$. The expansion (\ref{eq:expan}) has been used in \cite{part1}, \cite{part2},  \cite{orlando}, \cite{part3} to study analytically periodic problems.
\item Also treated in \cite{part2},  \cite{orlando}, \cite{part3} are quasiperiodic problems. If $\omega\in \Z^d$ is a vector of frequencies, these  are of the form (\ref{eq:ode}) with $A =\Z^d$
    and
\begin{equation}\label{eq:lambda}
    \lambda_\bk\!(t) = \exp(i\bk\cdot\omega t),\qquad
    \bk \in \Z^d.
\end{equation}
This case will be taken up in the next section.
\end{enumerate}

The notation in (\ref{eq:expan}) may be made slightly more compact by considering $A$ as an {\em alphabet}
and the strings $a_1\cdots a_n$ as {\em words} with $n$ letters. Then, if $\W_n$ represents the set of all words with
$n$ letters, (\ref{eq:expan}) reads
\[
 x(t) = x_0 + \sum_{n=1}^\infty \sum_{w\in \W_n}\alpha_w\!(t)\, f_w\!(x_0).
\]

If we furthermore introduce the empty word $\emptyset$ and set $\W_0 = \{\emptyset\}$, $\alpha_\emptyset =1$, $f_\emptyset(x) =x$, then
the last expansion becomes
\begin{equation}\label{eq:expanw}
 x(t) = \sum_{n=0}^\infty \sum_{w\in \W_n}\alpha_w\!(t)\, f_w\!(x_0) = \sum_{w\in \W}\alpha_w\!(t)\, f_w\!(x_0),
\end{equation}
where $\W$ represents the set of all words.

Subsequent developments will make much use of the  set $\C^\W$ of all mappings $\delta:\W\rightarrow \C$; if $\delta \in \C^\W$ and $w$ is a word, then $\delta_w$ is a complex number. This set is obviously a vector space for the usual operations between maps: if $\mu_1,\mu_2$ are scalars and $\delta_1,\delta_2\in\C^\W$, then  $(\mu_1 \delta_1+\mu_2 \delta_2)\in\C^\W$ is defined by $(\mu_1 \delta_1+\mu_2 \delta_2)_w = \mu_1 (\delta_1)_w+\mu_2 (\delta_2)_w$ for each $w\in\W$.

The expansion (\ref{eq:expanw}) motivates the following definition:

\begin{definition} If $\delta\in\C^\W$, then its
corresponding word series is the formal series
\begin{equation}\label{eq:ws}
W_\delta(x) = \sum_{w\in\W} \delta_w f_w\!(x).
\end{equation}
The scalars $\delta_w$ and the functions $f_w$ will be called the {\em coefficients} of the series and {\em word-basis functions} respectively.
\end{definition}

Clearly the word-basis functions  change with the mappings $f_a$ in the system
(\ref{eq:ode}) being studied.
With this terminology, for each fixed $t$, the formal series (\ref{eq:expanw}) for the solution $x(t)$ of (\ref{eq:ic})--(\ref{eq:ode}) is a word series whose coefficients $\alpha_w(t)$ are given by (\ref{eq:alpha}) and $\alpha_\emptyset(t)= 1$ (these coefficients are independent of the mappings $f_a$). Also, for each $t$, the right-hand side of (\ref{eq:ode}) is a word series with  coefficients $\beta_a(t)=\lambda_a(t)$ for words with one letter and $\beta_w(t)=0$ for all other words.
As we shall see later, word series $W_\delta$ corresponding to other choices of coefficients  $\delta_w$ are useful in the analysis of dynamical systems and their numerical integrators.

\begin{remark} Word series and moulds.\ \em In Ecalle's terminology, word series are {\em moulds}, see \cite{ecalle}, \cite{fm}, and the convolution product considered below is the mould product. (Ecalle \cite{ecalle} also considered a composition of moulds ---not discussed in this paper--- that is analogous
to the substitution of B-series \cite{substitution}.)
\end{remark}

\begin{remark} Word series and B-series.\ \em  Each $f_w$, $w\neq \emptyset$, is built up from partial derivatives of the $f_a$, $a\in A$, e.g., if $a,b,c \in A$,
\begin{eqnarray*}
f_{ba}(x) &=& \partial_xf_a(x)\, f_b(x),\\
f_{cba}(x) &=& \partial_xf_{ba}(x)f_c(x) = \partial_{xx} f_a(x)[f_b(x),f_c(x)]+\partial_xf_a(x)\,\partial_xf_b(x)\, f_c(x).
\end{eqnarray*}
The functions $\partial_xf_a(x)\, f_b(x)$, $\partial_{xx} f_a(x)[f_b(x),f_c(x)]$,$\partial_xf_a(x)\,\partial_xf_b(x)\, f_c(x)$ in these expressions are examples of {\em elementary differentials;} each word basis function $f_w(x)$, with $w\in\W_n$, $n >0$, is a linear combination with integer coefficients of elementary differentials of order $n$ (i.e.\ containing $n$ functions $f_a(x)$). There is an elementary differential corresponding to each $A$-coloured (or $A$-decorated) rooted tree, i.e.\ to each rooted tree where to each vertex it has been assigned  an element of $A$. By expanding each word basis function in terms of elementary differentials, the series (\ref{eq:ws}) becomes a so-called B-series
$$
\sum_{\tau} \Delta_\tau {\mathcal F}_\tau(x),
$$
where the summation is extended to all $A$-coloured  rooted trees and ${\mathcal F}_\tau(x)$ is the elementary differential corresponding to $\tau$.  B-series were introduced by Hairer and Wanner \cite{HW} in the simplest case where the alphabet $A$ has only one letter. B-series corresponding to this and larger alphabets are often used in numerical analysis;  word series being more compact are better suited to analyze some integrators. For  the relation between  word series and B-series   see \cite{anderfocm} and \cite{fm} (Ecalle used in this connection the terminology arborifaction-coarborification).
\end{remark}
\begin{remark}
\label{rem:eps}
 Word series as power series. \em For a system $(d/dt) x = \epsilon \sum_a \lambda_a(t) f_a(x)$, where $\epsilon$ is a scalar parameter, (\ref{eq:ws}) becomes the formal power series
$$
\sum_{n=0}^\infty \epsilon^n \!\sum_{w\in\W_n} \gamma_w f_w(x).
$$
Note that when the alphabet $A$ is infinite the coefficient of $\epsilon^n$ is itself an infinite series that has to be understood formally. The format (\ref{eq:ws}) is of course recovered from the power series by setting $\epsilon=1$; therefore both formats are equivalent.
While the papers \cite{part2}, \cite{part3} use the power series format, we prefer to work  with   (\ref{eq:ws}) as it leads to more compact formulae. Some readers may find it useful to mentally substitute $\epsilon f_a$ for $f_a$ everywhere in what follows; this may be particularly the case
for the perturbed integrable problems considered in Section~\ref{sec:extended}.
\end{remark}

\subsection{Operations with word series}

\subsubsection{The convolution product}

Given $\delta,\delta^\prime\in\C^\W$, we  associate with them its {\em convolution product} $\delta\star\delta^\prime\in\C^\W$ defined by
\begin{equation}\label{eq:convol}
(\delta\star\delta^\prime)_{a_1\cdots a_n} = \delta_\emptyset\delta^\prime_{a_1\cdots a_n}
+ \sum_{j=1}^{n-1} \delta_{a_1\cdots a_j}\delta^\prime_{a_{j+1}\cdots a_n}
+\delta_{a_1\cdots a_n}\delta^\prime_\emptyset
\end{equation}
(here it is understood that $(\delta\star\delta^\prime)_\emptyset = \delta_\emptyset\delta^\prime_\emptyset$). The convolution product is not commutative, but it is associative and has a unit (the element $\uno\in \C^\W$ with $\uno_\emptyset = 1$ and $\uno_w = 0$ for $w\neq \emptyset$).

As we shall see, the operation $\star$ plays an essential role in the manipulation of word series.

\subsubsection{The  group $\G$}
If $w\in\W_m$ and $w^\prime\in \W_n$ are words, $m,n\geq 1$, their {\em shuffle product} $w\sh w^\prime$ \cite{reu} is the formal sum of the $(m+n)!/(m!n!)$
words with $m+n$ letters that may be obtained by interleaving the letters of $w$ and $w^\prime$ while preserving the order in which the letters appear in each word. (Examples: for the words $ab$, $cd$, the shuffle product is
$ab \sh cd = abcd+acbd+cabd+acdb+cadb+cdab$, for the words $ab$, $a$, the product is $ab\sh a = aba+aab+aab$.)
In addition $\emptyset \sh w = w  \sh \emptyset= w$ for each $w\in\W$. The operation $\sh$ is commutative and associative and has word $\emptyset$ as a unit.

We denote by $\G$ the set of those $\gamma \in \C^\W$  that satisfy the so-called {\em shuffle relations:} $\gamma_\emptyset = 1$ and, for each $w,w^\prime\in \W$,
\begin{equation}\label{eq:defgr}
\gamma_w\gamma_{w^\prime} = \sum_{j=1}^N \gamma_{w_j}\qquad \mbox{\rm if}\qquad w\sh w^\prime = \sum_{j=1}^N w_j.
\end{equation}
The set $\G$ with the operation $\star$ may be regarded in a formal sense (cf. \cite{geir}) as a non-commutat\-ive  Lie group (see  Section~\ref{sss:shuffle}). For each fixed $t$, the family of coefficients defined by (\ref{eq:alpha}) and $\alpha_\emptyset(t) =1$ is an element of the group $\G$ (to prove this, consider (\ref{eq:defgr}) for $w\in\W_m$ and $w^\prime \in \W_n$ and use (\ref{eq:alpha}) to write $\alpha_w\alpha_{w^\prime}$ and each $\alpha_{w_j}$ as integrals over subsets of $\R^{m+n}$, cf. \cite[Corollary 3.5]{reu}).

When $\gamma$ belongs to $\G$, the word series $W_\gamma(x)$ has properties that are not shared by general word series. For $\gamma \in \G$, changes of variables $x = C(X)$ commute with the formation of word series as described in \cite[Proposition 3.1]{part2}.
Moreover, for $\gamma\in\G$, $W_\gamma(x)$ may be substituted in an arbitrary word series $W_\delta(x)$, $\delta\in\C^\W$, to get a new word series; more precisely
\begin{equation}\label{eq:act}
W_\delta\big (W_{\gamma}(x)\big) = W_{\gamma\star \delta}(x),
\end{equation}
i.e.\ the coefficients of the word series resulting from the substitution are given by the convolution product $\gamma\star\delta$ (this is proved in Section~\ref{sss:action}). A similar rule exists of course for B-series, but the recipe there is more complicated than (\ref{eq:convol}) \cite{HW}, \cite[Chapter III]{hlw}.

In the numerical analysis of differential equations, word series with coefficients in $\G$ appear e.g.\ as expansions in powers of the stepsize of splitting integrators, see Section~\ref{sec:split}. Then (\ref{eq:act}) provides the recipe to compose integrators or to compose an integrator and a mapping. For each fixed $t$, the right-hand side of (\ref{eq:ode}) is an example of a word-series with coefficients in the Lie algebra $\g$ of the group $\G$ that we study next.

\subsubsection{The Lie algebra $\g$}

We denote $\g$ the set of elements $\beta\in\C^\W$ such that $\beta_\emptyset = 0$ and for each pair of nonempty words $w,w^\prime$,
\[
\sum_{j=1}^N \beta_{w_j} = 0\qquad \mbox{\rm if}\qquad w\sh w^\prime = \sum_{j=1}^N w_j.
\]
It is clear that $\g$ is a vector subspace of the vector space $\C^\W$; furthermore $\g$ is closed for the skew-symmetric product defined by
\begin{equation}\label{eq:bracket}
[\beta,\beta^\prime] = \beta\star\beta^\prime-\beta^\prime\star\beta.
\end{equation}
This product satisfies the Jacobi identity and therefore endows $\g$ with a structure of Lie algebra (see Section~\ref{sss:shuffle}). In fact $\g$ is the Lie algebra of the Lie group $\G$: the elements $\beta\in\g$ coincide with the velocities at $\uno\in \G$ of curves in $\G$. In symbols, if
$\gamma(t)$, $t\in\R$ is a curve in $\G$ such that $\gamma(0) = \uno$, then $\beta\in\C^\W$ defined by
\begin{equation}\label{eq:derivada}
\beta = \left.\frac{d}{dt} \gamma(t)\right|_{t=0}
\end{equation}
(i.e.\ $\beta_w= (d/dt) \gamma_w(0)$ for
for each $w\in\W$) belongs to $\g$. Moreover any $\beta\in\g$ arises in this way: the exponential
\begin{equation}\label{eq:exp}
\exp_\star(t\beta) = \uno +\sum_{j=1}^\infty \frac{t^j}{j!} \beta^{\star j}
\end{equation}
($\beta^{\star j}$ is the convolution product of $j$ factors all equal to $\beta$) defines a curve of elements of $\G$ with velocity $\beta$ at $t=0$. The points of this curve actually form a one-parameter subgroup of $\G$ since $\exp_\star(t\beta)\star \exp_\star(t^\prime\beta) = \exp_\star((t+t^\prime)\beta)$.
The exponent $\beta$ may be retrieved from the exponential $\gamma = \exp_\star(\beta)$ by means of the logarithm
\begin{equation}\label{eq:log}
\beta = \log_\star(\gamma) = \sum_{j=1}^\infty \frac{(-1)^{(j+1)}}{j} (\gamma-\uno)^{\star j}.
\end{equation}

Just as the convolution product with an element of $\G$ corresponds to the operation of substitution of the associated word series (see (\ref{eq:act})), the convolution bracket (\ref{eq:bracket}) corresponds to the Jacobi bracket (commutator) of the associated word series, for $\beta,\beta^\prime\in \g$:
\[
\big(\partial_x W_{\beta^\prime}(x)\big) W_\beta(x) - \big(\partial_x W_\beta(x)\big) W_{\beta^\prime}(x)= W_{[\beta,\beta^\prime]}(x).
\]
To prove this, let $\gamma(t)$ be a curve with velocity $\beta$ as above; then, for any $\delta\in\C^\W$,
\begin{equation}\label{eq:calculote}
(\partial_x W_{\delta}(x))W_\beta(x)
=  \left.\frac{d}{dt} W_{\delta}\big(W_{\gamma(t)}(x)\big)\right|_{t=0}
=\left.\frac{d}{dt} W_{\gamma(t)\star\delta}(x)\right|_{t=0}
= W_{\beta\star\delta}(x).
\end{equation}
(We have successively used  the chain rule, (\ref{eq:act}), and the bilinearity of $\star$.)

Since for $\beta\in\g$, the word series $W_\beta(x)$ belongs to the Lie algebra (for the Jacobi bracket) generated by the mappings (vector fields) $f_a$, the Dynkin-Specht-Wever formula \cite{jac} may be used to rewrite the word series in terms of iterated commutators of these mappings:
\begin{equation}\label{eq:dynkin}
W_\beta(x)
 = \sum_{n=1}^\infty \frac{1}{n} \sum_{a_1,\dots,a_n\in A}
 \beta_{a_1\cdots a_n} [[\cdots[[f_{a_1},f_{a_2}],f_{a_3}]\cdots],f_{a_n}](x).
\end{equation}
(For $n=1$ the terms in the inner sum are of the form $\beta_{a_1} f_{a_1}(x)$.)

\subsubsection{Nonautonomous differential equations in $\G$}
Initial value problems
\begin{equation}\label{eq:ode3}
\frac{d}{dt} x(t) = W_{\beta(t)}(x(t)),\qquad x(0) =x_0,
\end{equation}
where for each $t$, $\beta(t)\in\g$, are a natural generalization of (\ref{eq:ic})--(\ref{eq:ode}) (we recall that the right-hand side (\ref{eq:ode})  does not include contributions from basis functions associated with words with more than one letter). These problems may be solved formally by using the ansatz $x(t) = W_{\alpha(t)}(x_0)$ with $\alpha(t)\in\G$ for each $t$. In view of (\ref{eq:act}), we may write
$$
\frac{d}{dt}W_{\alpha(t)}(x_0) = W_{\beta(t)}(W_{\alpha(t)}(x_0)) = W_{\alpha(t)\star\beta(t)}(x_0),\qquad W_{\alpha(0)}(x_0) = x_0,
$$
which leads to the linear, nonautonomous initial value problem
\begin{equation}\label{eq:odebeta}
\frac{d}{dt} \alpha(t) = \alpha(t)\star\beta(t),\qquad \alpha(0) = \uno.
\end{equation}
For the empty word, according to (\ref{eq:convol}), $(d/dt) \alpha_\emptyset = \alpha_\emptyset(t) \beta_\emptyset(t)$; since $\beta(t)\in\g$ implies $\beta_\emptyset(t) =0$, we see that $\alpha_\emptyset(t) =1$.
For a word $a\in\W_1$, $(d/dt)\alpha_a(t) = \alpha_\emptyset(t)\beta_a(t) + \alpha_a(t)\beta_\emptyset(t)$, which leads to $\alpha_a(t) = \int_0^t \beta_a(t_1)\, dt_1$. The process may be continued in an obvious way and
induction on the number of letters shows that (\ref{eq:odebeta}) uniquely determines $\alpha_w(t)$ for each $w\in\W$. Furthermore, for each $t$, the element $\alpha(t)\in\C^\W$ defined in this way belongs to $\G$; while this may be established by means of the Magnus expansion (see e.g.\ \cite{magnus}, \cite[Chapter IV]{hlw}) we  provide an elementary proof in Section~\ref{s:technical}.

Conversely, any   curve $\alpha(t)$ of group elements with $ \alpha(0) = \uno$ solves a problem of the form (\ref{eq:odebeta}) with
$$
\beta(t) = \alpha(t)^{-1} \star \left(\frac{d}{dt} \alpha(t)\right).
$$
Since
$$
\alpha(t)^{-1} \star \left(\frac{d}{dt} \alpha(t)\right) = \left. \frac{d}{ds} \Big(\alpha(t)^{-1} \star \alpha(t+s)\Big)\right|_{s=0},
$$
for each $t$, the element $\beta(t)$ defined in this way is a member of $\g$.

The investigation of normal forms below is based on changing variables.
A change of variables $x = W_{\kappa}(X)$, $\kappa\in\G$, transforms the problem (\ref{eq:ode3}) into
$$
\frac{d}{dt} X(t) = W_{B(t)}(X(t)), \qquad X(0) = X_0,
$$
with $B(t)\star \kappa = \kappa \star\beta(t)$ (or $B(t) = \kappa \star\beta(t)\star\kappa^{-1}$), $X_0 = W_{\kappa^{-1}}(x_0)$ ($\kappa^{-1}$ is the inverse of $\kappa$ in the group $\G$); this is a direct consequence of (\ref{eq:act}) and (\ref{eq:calculote}).

\subsubsection{The Hamiltonian case}
Consider now the particular case where the dimension $D$ of (\ref{eq:ode}) is even and each $f_a(x)$ is a Hamiltonian vector field \cite{ssc}, i.e.\ $f_a(x) = J^{-1}\nabla H_a(x)$, where $J^{-1}$ is the standard symplectic matrix. Recall \cite{arnoldmec} that the Jacobi bracket (commutator) $[J^{-1}\nabla A, J^{-1}\nabla B] $ of two Hamiltonian vectors fields is again a Hamiltonian vector field and that the corresponding Hamiltonian function is
the Poisson bracket of the Hamiltonians $A$ and $B$, defined by $\{A,B\}(x) = \nabla A(x)^T J^{-1} B(x)$. According to (\ref{eq:dynkin}), for each $\beta\in\g$, the vector field $W_\beta(x)$ is Hamiltonian
$$
W_\beta(x) = J^{-1}\nabla\Ham_\beta(x)
$$
with Hamiltonian function
$$
\Ham_\beta(x) = \sum_{w\in\W,\, w\neq\emptyset} \beta_w H_w(x),
$$
where, for each nonempty word $w=a_1\cdots a_n$,
\begin{equation}\label{eq:wordham}
H_w(x) = \frac{1}{n}\{\{\cdots\{\{H_{a_1},H_{a_2}\},H_{a_3}\}\cdots\},H_{a_n}\}(x)
.
\end{equation}

For Hamiltonian systems, changes of variables $x = W_{\kappa}(X)$, $\kappa\in\G$, are canonically symplectic; after the change of variables the system is again Hamiltonian and the new Hamiltonian function is obtained by changing variables in the old Hamiltonian function \cite{arnoldmec}.

\section{Extended word series}
\label{sec:extended}

In this section we adapt the preceding material to cover perturbed integrable problems.

\subsection{Perturbed integrable problems}

We now consider systems of the form
$$
\frac{d}{dt} \left[ \begin{matrix}y\\ \theta\end{matrix}\right]
= \left[ \begin{matrix}0\\ \omega\end{matrix}\right]
+f(y,\theta),
$$
where $y\in\R^{D-d}$, $0<d\leq D$, $\omega\in\R^d$ is a vector of frequencies $\omega_j>0$, $j = 1,\dots, d$, and $\theta$ comprises $d$ angles, so that $f(y,\theta)$ is $2\pi$-periodic  in each component of $\theta$ with Fourier expansion
%\begin{equation}\label{eq:fourier}
$$
f(y,\theta) = \sum_{\bk \in\Z^d} \exp(i \bk\cdot \theta)\: \hat f_\bk(y)
$$
%\end{equation}
($\hat f_\bk(y)$ and $\hat f_{-\bk}(y)$ are mutually conjugate, so as to have a real problem). Systems of this form appear in many applications, perhaps after a change of variables (see the Appendix). When $f\equiv 0$ the system is integrable (the angles rotate with uniform angular velocity and $y$ remains constant) and accordingly we refer to problems of this class as perturbed integrable problems and to $f$ as the perturbation (some readers may prefer to substitute $\epsilon f$ for $f$, see Remark~\ref{rem:eps}).

After introducing the  functions
\begin{equation}\label{eq:fbk}
f_\bk(y,\theta) = \exp(i\bk\cdot \theta)\: \hat f_\bk(y),\qquad y\in\R^{D-d},\: \theta\in\R^d,
\end{equation}
that satisfy the fundamental identity
\begin{equation}\label{eq:fund}
f_\bk(y,\theta_1+\theta_2) = \exp(i \bk\cdot\theta_1) \: f_\bk(y,\theta_2),
\end{equation}
the system takes the form
\begin{equation}\label{eq:ode2}
\frac{d}{dt} \left[ \begin{matrix}y\\ \theta\end{matrix}\right]
= \left[ \begin{matrix}0\\ \omega\end{matrix}\right]
+f(y,\theta)
= \left[ \begin{matrix}0\\ \omega\end{matrix}\right]
+\sum_{\bk \in\Z^d} f_\bk(y,\theta).
\end{equation}
To find the solution with initial conditions
\begin{equation}\label{eq:ic2}
y(0) = y_0,\qquad \theta(0) = \theta_0,
\end{equation}
we perform the time-dependent change of variables $\theta =\eta+t\omega $ to get

\begin{equation}\label{eq:odeeta}
\frac{d}{dt} \left[ \begin{matrix}y\\ \eta\end{matrix}\right]
= \sum_{\bk \in\Z^d} \exp(i \bk\cdot \omega t)\:f_\bk(y,\eta),
\end{equation}
a particular instance of the problem considered in Section \ref{sec:words}. The alphabet $A$ coincides with $\Z^d$,
and, for each \lq letter\rq\ $\bk$, $\lambda_\bk(t)$  is given by (\ref{eq:lambda}). The formula (\ref{eq:expan}) yields
\begin{equation}\label{eq:solucionprev}
\left[ \begin{matrix}y(t)\\ \eta(t)\end{matrix}\right]
=
\left[ \begin{matrix}y(0)\\ \eta(0)\end{matrix}\right]
+\sum_{n=1}^\infty \sum_{\bk_1,\dots,\bk_n}\alpha_{\bk_1\cdots \bk_n}(t)\, f_{\bk_1\cdots \bk_n}(y(0),\eta(0)),
\end{equation}
where the coefficients $\alpha$ are still given by (\ref{eq:alpha}) (but recall that now the letters
$a$ are multiindices $\bk$) and the word basis functions are defined by (\ref{eq:fbk}) and (\ref{eq:wbf}) (the Jacobian in (\ref{eq:wbf}) is taken with respect to the $D$-dimensional variable $(y,\theta)$). We conclude that, in the original variables, the solution flow of (\ref{eq:ode2}), has the formal expansion
\begin{equation}\label{eq:solucion}
\phi_t(y_0,\theta_0) =
\left[ \begin{matrix}y(t)\\ \theta(t)\end{matrix}\right] =
\left[ \begin{matrix}y_0\\ \theta_0\end{matrix}\right]
+\left[ \begin{matrix}0\\t \omega \end{matrix}\right]
+\sum_{n=1}^\infty \:\sum_{\bk_1,\dots,\bk_n}\alpha_{\bk_1\cdots \bk_n}(t)\, f_{\bk_1\cdots \bk_n}(y_0,\theta_0).
\end{equation}
Note that the word basis functions are {\em independent of the frequencies} $\omega$ and the coefficients $\alpha$ are {\em independent of $f$}. Also from
(\ref{eq:fund}) we have the identity:
\begin{equation}\label{eq:fund2}
f_{\bk_1\cdots \bk_n}(y,\theta_1+\theta_2) = \exp(i (\bk_1+\cdots+\bk_n)\cdot \theta_1)\: f_{\bk_1\cdots \bk_n}(y,\theta_2),
\end{equation}
and, in particular
\begin{equation}\label{eq:fund2bis}
f_{\bk_1\cdots \bk_n}(y,\theta) = \exp(i (\bk_1+\cdots+\bk_n)\cdot \theta)\: f_{\bk_1\cdots \bk_n}(y,0),
\end{equation}

With the notation of Section \ref{sec:words}, we write (\ref{eq:solucion}) in the following form (here and later $x = (y,\theta)$):
$$
x(t) = \left[ \begin{matrix}0\\t \omega \end{matrix}\right] +W_{\alpha(t)}(x_0).
$$
 In order to make the formula even more compact, we introduce the vector space  $\calC = \C^d\oplus \C^\W$
and define:
\begin{definition} If $(v,\delta)\in\calC$, then its
corresponding extended word series is the formal series

\[
\overline{W}_{(v,\delta)}(x) =  \left[ \begin{matrix}0\\ v \end{matrix}\right] +\sum_{w\in\W} \delta_w f_w\!(x).
\]
\end{definition}
Then the solution (\ref{eq:solucion}) of (\ref{eq:ode2})--(\ref{eq:ic2}) has the expansion
$$x(t) = \overline{W}_{(t\omega,\alpha(t))}(x_0),\qquad (t\omega,\alpha(t)) \in \calC,$$
with  $\alpha(t)\in \G \subset \C^\W$  as defined in Section~\ref{sec:words}. Also the right-hand side of (\ref{eq:ode2}) is an extended word series $\overline{W}_{(\omega,\beta)}(x)$ with $\beta\in\g\subset \C^\W$ defined by
\begin{equation}\label{eq:b}
\beta_w = 1 \quad \mbox{if}\quad w\in\W_1,\qquad \beta_w = 0 \quad\mbox{if}\quad w\notin \W_1.
\end{equation}

\subsection{Operations with extended word series}

\subsubsection{The operation $\bigstar$}

The following two linear operators will appear repeatedly.
If $v$ is a $d$-vector,
 $\Xi_{v}$ is the linear operator in $\C^\W$ that maps each $\delta \in \C^\W$ into the element of $\C^\W$ defined by
 $(\Xi_{v} \delta)_\emptyset = \delta_\emptyset$ and
\begin{equation}\label{eq:Xi}
(\Xi_{v}\delta)_w =\exp( i (\bk_1+\cdots+\bk_n)\cdot v)\: \delta_w.
\end{equation}
 for $w = \bk_1\dots\bk_n$.
The linear operator $\xi_{v}$ on $\C^{\W}$ is defined as follows:
 $(\xi_{v} \delta)_{\emptyset}=0$, and for each word $w=\bk_1 \cdots \bk_n$,
\begin{equation*}
(\xi_{v} \delta)_{w} = i  (\bk_1 +\cdots + \bk_n)\cdot v\: \delta_{w}.
\end{equation*}
Thus $\Xi_{v}$ and $\xi_{v}$ are  {\em diagonal} operators with eigenvalues $\exp( i (\bk_1+\cdots+\bk_n)\cdot v)$ and
 $i  (\bk_1 +\cdots + \bk_n)\cdot v$ respectively.
Observe that $\Xi_{v}(\gamma \star \delta) = (\Xi_{v} \gamma) \star (\Xi_{v} \delta)$ if $\gamma, \delta\in \C^{\W}$ and that:
\begin{equation*}
 \frac{d}{dt} \Xi_{t v} =
 \Xi_{t v} \xi_{v} = \xi_{v}\Xi_{t v}.
\end{equation*}

The symbol $\overline{\G}$ denotes the subset of $\calC$ comprising the elements $(u,\gamma)$ with $u \in \C^d$ and $\gamma\in\G$. For each $t$, the
solution coefficients $(t\omega,\alpha(t))\in\calC$ found above provide an example of element of $\overline{\G}$.
With the help of $\Xi_u$ we define an operation $\bigstar$ as follows.
 If $(u,\gamma)\in\overline{\G}$ and $(v,\delta) \in\calC$, then
 $$
(u,\gamma) \bigstar(v,\delta) = (\gamma_\emptyset v + \delta_{\emptyset} u,
\gamma \star (\Xi_{u} \delta))\in\calC.
$$
By using (\ref{eq:act}) and (\ref{eq:fund2}), it is a simple exercise to check that  $\overline{\G}$ acts by substitution on extended word series as follows:
\begin{equation}\label{eq:act2}
\overline{W}_{(v,\delta)}\big(\overline{W}_{(u,\gamma)}(x)\big) =
\overline{W}_{(u,\gamma)\bigstar (v,\delta)}(x), \qquad \gamma\in{\G}.
\end{equation}
In fact we have defined the operation $\bigstar$  so as to ensure this property.
The set $\overline{\G}$ is a group for the product $\bigstar$ and $\C^d$ and $\G$ may be viewed as subgroups of $\overline{\G}$.\footnote{
Consider the group homomorphism  from the additive group $\C^d$ to the group of automorphisms of $\G$ that maps each $\mu \in \C^d$ into $\Xi_{\mu}$. Then $\overline{\G}$ is the (outer) semidirect product of $\G$ and the additive group $\C^d$ with respect to this homomorphism.}   The unit of $\overline{\G}$ is the element $\overline\uno = (0,\uno)$.

\subsubsection{The Lie algebra $\overline{\g}$}
\label{sss:tlag}

As a set, the Lie algebra $\overline{\g}$ of the group $\overline{\G}$ consists of the elements
$(v,\delta)\in\calC$ with $\delta\in\g$. Let us describe the bracket in $\overline{\g}$.
Given $v \in \C^d$ and $\delta \in \g$, we have the trivial decomposition
\begin{equation*}
\overline{W}_{(v,\delta)} = \left[
  \begin{array}{c}
    0 \\ v
  \end{array}
\right] + W_{\delta}(x) =
\overline{W}_{(v,0)}+\overline{W}_{(0,\delta)}.
\end{equation*}
By using (\ref{eq:fund2bis}), one can check that the Jacobi bracket of the vector fields $\overline{W}_{(\lambda,0)}$ and
$\overline{W}_{(0,\delta)}$ is
\begin{equation*}
[\overline{W}_{(v,0)},\overline{W}_{(0,\delta)}] = \overline{W}_{(0,\xi_{v} \delta)}.
\end{equation*}
From these relations we conclude that, for arbitrary  $(v,\delta),(u,\eta) \in \C^d\oplus\g$, the Jacobi bracket of the vector fields $\overline{W}_{(v,\delta)}$, $\overline{W}_{(u,\eta)}$ is given by
\begin{equation*}
[\overline{W}_{(v,\delta)},\overline{W}_{(u,\eta)}] = \overline{W}_{(0,\xi_{v}\eta-\xi_{u} \delta + \delta\star\eta -\eta\star\delta )}.
%\overline{W}_{[(v,\delta),(u,\eta))]}(x),
\end{equation*}
Accordingly the bracket of $\overline{\g}$ has the expression
\begin{equation*}
[(v,\delta),(u,\eta))] =
(0,\xi_{v}\eta-\xi_{u} \delta + \delta\star\eta -\eta\star\delta).
\end{equation*}
The $0$  reflects  the fact that $\C^d$ is an Abelian subgroup of $\overline{\G}$.

\subsubsection{Nonautonomous differential equations in $\overline{\G}$}
\label{sss:eqnsoverlineG}
The initial value problem
$$
\frac{d}{dt} x(t) = \overline{W}_{(\omega,\beta(t))}(x(t)),\qquad x(0) =x_0,
$$
where $(\omega,\beta(t))\in\overline{\g}$ for each $t$, may be formally solved in a manner that is exactly parallel to treatment given above to
(\ref{eq:ode3}): $x(t)=\overline{W}_{(t\omega,\alpha(t))}(x_0)$, where $\alpha(0) = \uno$ and
\begin{equation*}
\frac{d}{dt} (t\omega,\alpha(t)) = (t\omega,\alpha(t)) \bigstar (\omega,\beta(t)).
\end{equation*}
Observe that the right-hand side of this equation is, by definition of $\bigstar$, equal to $(t\omega, \alpha(t) \star (\Xi_{t\omega} \beta(t)))$, so that $\alpha(t)$ is the solution of an initial value problem of the form (\ref{eq:odebeta}) with $\beta(t)$ replaced by $\Xi_{t\omega} \beta(t)$.

A change of variables
$x = \overline{W}_{(v,\kappa)}(X)$, $(v,\kappa)\in\overline{\G}$, can be seen to transform the  problem  into
$$
\frac{d}{dt} X(t) =\overline{ W}_{(\omega,B(t))}(X(t)), \qquad X(0) = X_0,
$$
where now $B(t)$ is determined from
\begin{equation*}
B(t)\star \kappa + \xi_{\omega}\kappa= \kappa \star (\Xi_{v} \beta(t) )
\end{equation*}
and, of course, $x_0 =\overline{W}_{(v,\kappa)}(X_0)$ or $X_0 =\overline{W}_{(v,\kappa)^{-1}}(x_0)$. (Note that $(v,\kappa)^{-1} = (-v,\Xi_{-v}\kappa^{-1})$.)

\subsubsection{Perturbed Hamiltonian problems}
To end this section, assume  in (\ref{eq:ode2}), that the dimension $D$ is even with  $D/2- d=m\geq0$ and that the vector of unknowns takes the form
$$
x = (y,\theta) =(p^1,\dots,p^m; q^1,\dots,q^m;a^1,\dots,a^d;\theta^1,\dots,\theta^d),
$$
where $p^j$ is the momentum canonically conjugate to the co-ordinate $q^j$ and $a^j$ is the momentum (action) canonically conjugate to the co-ordinate (angle) $\theta^j$. If each $f_\bk(x)$ in (\ref{eq:fbk}) is a Hamiltonian vector field with Hamiltonian function $H_\bk(x)$, then the system (\ref{eq:ode2}) is itself Hamiltonian for the Hamiltonian function
$$
\sum_{j=1}^d \omega_j a^j +\sum_{\bk\in\Z^d} H_\bk(x).
$$

For each $(\omega,\beta)\in\overline{\g}$, the extended word series $\overline{W}_{(\omega,\beta)}(x)$ is a  Hamiltonian formal vector field, with Hamiltonian function
\begin{equation}
\label{eq:Hamseries}
\sum_{j=1}^d \omega_j a^j + \sum_{w\in\W,\,w\neq\emptyset}\beta_w H_w,
\end{equation}
with $H_w(x)$ as in (\ref{eq:wordham}).
 Note that the Lie bracket in $\overline{\g}$ can be used to compute the Poisson bracket of formal Hamiltonian functions of the form (\ref{eq:Hamseries}).

\section{Normal forms and averaging}
\label{s:normal}

In this section we show how the algebraic machinery introduced above may be applied to build a theory of normal forms \cite{arnoldode}, \cite{aver}
for the perturbed integrable problems of the form (\ref{eq:ode2}). This theory hinges on the fact that the linear operator
$\overline{W}_{(0,{\delta})} \mapsto [\overline{W}_{(\omega,0)}, \overline{W}_{(0,{\delta})}]$ ($\delta\in\g$) coincides, as we have seen in Section~\ref{sss:tlag}, with the {\em  diagonal} operator $\overline{W}_{(0,{\delta})} \mapsto \overline{W}_{(0,{\xi_\omega \delta})}$.

Let us consider  an autonomous system
\begin{equation}\label{eq:systnormal}
\frac{d}{dt}x =
\frac{d}{dt} \left[ \begin{matrix}y\\ \theta\end{matrix}\right]
= \left[ \begin{matrix}0\\ \omega\end{matrix}\right]
+
 W_{\beta}(x) = \overline{W}_{(\omega, \beta)}(x),\qquad \beta \in \g.
\end{equation}
As noted before, this format yields the
perturbed problem (\ref{eq:ode2}) when $\beta$ is chosen as in (\ref{eq:b}). The more general case where $\beta$ is any element in $\g$ will be necessary to deal with splitting integrators later.
We shall change variables $x = W_\kappa(X) = \overline{W}_{(0,\kappa)}(X)$, $\kappa\in\G$, in order to simplify
(\ref{eq:systnormal}) as much as possible.

\begin{remark}\label{rem:generalized}\em There is nothing lost by assuming that $x= W_\kappa(X)$ is a (not extended) word series in the new variables $X$ ---or equivalently an extended word series of the special format $x= \overline{W}_{(0,\kappa)}(X)$---. More general  changes $x=\overline{W}_{(v,\kappa)}(X)$ do not allow for additional simplications in (\ref{eq:systnormal}).
\end{remark}

From Section~\ref{sss:eqnsoverlineG}, we know that
the transformed system is
\begin{equation}\label{eq:systnormaltrans}
\frac{d}{dt}X =
\frac{d}{dt} \left[ \begin{matrix}Y\\ \Theta\end{matrix}\right]
= \left[ \begin{matrix}0\\ \omega\end{matrix}\right]
+
 W_{\widehat{\beta}}(X) = \overline{W}_{(\omega, \widehat{\beta})}(X),
\end{equation}
with
\begin{equation}\label{eq:normaleqn}
\xi_\omega \kappa +\widehat{\beta} \star \kappa = \kappa \star\beta.
\end{equation}

Our aim is to  choose  $\widehat{\beta}\in\g$ and $\kappa\in \G$ subject to (\ref{eq:normaleqn}) and such that
$\widehat{\beta}$ is as simple as possible; then the system is said to have been brought to normal form. Of course the maximum simplification would be obtained by setting $\widehat{\beta} = 0$, but for this choice of $\widehat{\beta}$ it is not possible to find an appropriate
$\kappa$; this will be clear in the proof of Theorem~\ref{th:theoremnormal} and is to be expected from general results on normal forms \cite{arnoldode}, \cite{aver}. More precisely, perturbations that commute with $\overline{W}_{(\omega,0)}$  cannot be eliminated by changing variables. One then has to restrict the attention to $\widehat{\beta}\in\g$ such that in (\ref{eq:systnormaltrans}) the unperturbed vector field and the perturbation commute, i.e.\ $[\overline{W}_{(\omega,0)}, \overline{W}_{(0,\widehat{\beta})}]=0$. This is equivalent to $\overline{W}_{(0,\xi_{\omega} \widehat{\beta})}=0$, or, in terms of the coefficients,
\begin{equation}\label{eq:condnormal}
i (\bk_1+\cdots+\bk_n)\cdot \omega\:\widehat{\beta}_{\bk_1\dots\bk_n} = 0,
\end{equation}
for each nonempty word $\bk_1\dots\bk_n$. We have then the following result, which is proved constructively in Section~\ref{ss:proofnormal}.

\begin{theorem}
\label{th:theoremnormal}
There is a change of variables $x = W_\kappa(X)$, $\kappa\in\G$, that reduces the system (\ref{eq:systnormal})
to the form (\ref{eq:systnormaltrans}), where $\widehat{\beta}\in\g$ and $ \widehat{\beta}_w = 0$ for all words $w=\bk_1\dots\bk_n$ such that $(\bk_1+\cdots+\bk_n)\cdot\omega \neq 0$. Furthermore the vector fields
$\overline{W}_{(\omega,0)}(X)$ and $\overline{W}_{(0, \widehat \beta)}(X)$ commute  and the solutions of (\ref{eq:systnormaltrans}) satisfy
\[
X(t) = \phi_t \Big( X(0)+
\left[ \begin{matrix}0\\ t\omega\end{matrix}\right]
\Big)= \phi_t( X(0))+\left[ \begin{matrix}0\\ t\omega\end{matrix}\right],
\]
where $\phi_t$ is the solution flow of the system $(d/dt)X = W_{\widehat{\beta}}(X)$. Equivalently,
$X(t) = \overline{W}_{(t\omega,\widehat{\alpha}(t))}(X(0))$, where
\[
(t \omega,\widehat{\alpha}(t)) = (t\omega, \exp_\star(t\widehat{\beta})) = \exp_\star(t\widehat{\beta}) \bigstar (t\omega,\uno) =
(t\omega,\uno) \bigstar  \exp_\star(t\widehat{\beta}).
\]

 If the system (\ref{eq:systnormal}) is Hamiltonian, the change of variables is canonical symplectic and the  transformed system (\ref{eq:systnormaltrans}) is Hamiltonian.
\end{theorem}

When $\omega$ is {\em nonresonant}, i.e.\ $\bk\cdot \omega \neq 0$ for $\bk\neq {\bf 0}$, the theorem implies, in view of (\ref{eq:fund2bis}), that the transformed vector field $\overline{W}_{(\omega,\widehat{\beta})}(X)$ is independent of the angular variables $\Theta$. In other words,  (\ref{eq:systnormaltrans}) is a system where the angles have been {\em averaged} \cite{arnoldode}, \cite{arnoldmec}, \cite{aver}.
In the general situation with a nontrivial resonant module
\[
{\mathcal M}_\omega = \{\bk\in\Z^d: \bk\cdot\omega=0\},
\]
the transformed vector field depends on $\Theta$. However  this dependence is only through a number of combinations
$\bl_1\cdot \Theta$,\dots, $\bl_r\cdot \Theta$, $r<d$, where $\bl_1$, \dots, $\bl_r\in\Z^d$ are linearly independent and span the resonant module.

\begin{remark}\em Consider the {\em highly oscillatory} case where  (\ref{eq:systnormal}) depends on a small parameter $\delta$ and $\omega= \mathcal{O}(1/\delta)$, $W_{\beta}(x)=\mathcal{O}(1)$.
The combinations  not eliminated by the change of variables have the property that their velocities $(d/dt)\bl_i\cdot \Theta$ are $\mathcal{O}(1)$, as distinct from the situation for the original angles with $(d/dt) \theta = \mathcal{O}(1/\delta)$. In this sense, the {\em fast} angles have been averaged when forming (\ref{eq:systnormaltrans}).
\end{remark}

For convenience, we shall use the expression {\em oscillatory word} to refer to those words $\bk_1\dots\bk_n$ for which
$(\bk_1+\cdots +\bk_n)\cdot\omega \neq 0$. Thus  the theorem may be rephrased as saying that the contributions to the vector field corresponding to oscillatory words may be removed from  (\ref{eq:systnormal}) by means of a change of variables. Note that the set $\g_0$ of all $\widehat{\beta}\in\g$ such that $\widehat{\beta}_w=0$ for all oscillatory words is a Lie subalgebra of $\g$; this follows from the fact that $\widehat{\beta}$ is in $\g_0$ if and only if
$(\omega,0)$ and $(0,\widehat{\beta})$ commute.

If we now express the commuting vector fields $\overline{W}_{(\omega,0)}(X)$ and $\overline{W}_{(0, \widehat \beta)}(X)$
in terms of the original variables $x$
by applying the recipe for changing variables given in Subsection~\ref{sss:eqnsoverlineG},
we obtain a decomposition of the right-hand side of (\ref{eq:systnormal}) as a commuting sum of two terms:
\[\overline{W}_{(\omega, \beta)}(x) =
 \overline{W}_{(\omega,\kappa^{-1} \star \xi_{\omega} \kappa)}(x)+
 W_{\kappa^{-1}\star \widehat{\beta}\star\kappa}(x).
\]
The second of these generates a flow $$x(t) = W_{\kappa^{-1}\star\exp_\star(t\widehat{\beta})\star\kappa}(x(0))$$ where the motion of the fast angles has been averaged. The former generates a quasi\-per\-iodic flow
$$x(t) =\overline{W}_{(0,\kappa^{-1})\bigstar(t\omega,\uno)\bigstar(0,\kappa)}(x(0))= \overline{W}_{(t\omega,\kappa^{-1}
\star \Xi_{t \omega} \kappa)}(x(0)).$$
Finally, the flow of (\ref{eq:systnormal}) is  given by
\[x(t) = \overline{W}_{(t\omega, \kappa^{-1}\star\exp_\star(t\widehat{\beta})\star\kappa\star\Xi_{t\omega})}(x(0)).\]

\begin{remark} \em In the nonresonant case, the commuting decomposition of $\overline{W}_{(\omega, \beta)}(x)$ has been obtained in \cite[Theorem 5.5]{part2}  by means of a different (but related) technique.
\end{remark}
\section{Splitting methods}
\label{sec:split}

Splitting algorithms \cite{ssc}, \cite{hlw} are natural candidates to integrate perturbed integrable problems. In this connection, it is extremely important to emphasize that the practical implementation of splitting methods is {\em not} necessarily based on the simple format (\ref{eq:ode2}). Such a simple format is typically  reached after  suitable changes of variables and is quite convenient for the analysis.
These points are discussed in the Appendix.

 Given real coefficients, $a_j$ and $b_j$, $j=1,\dots,r$, we study the splitting integrator for (\ref{eq:ode2}) defined by
\begin{equation}\label{eq:integrat}
\widetilde{\phi}_h = \phi^{(P)}_{b_rh}\circ \phi^{(U)}_{a_rh}\circ\cdots\circ \phi^{(P)}_{b_1h}\circ \phi^{(U)}_{a_1h}.
\end{equation}
Here $h$ is the step-length, $\widetilde{\phi}_h$ the mapping in $\R^D$ that advances the numerical solution over one time step, and $\phi_t^{(U)}$ and
$\phi_t^{(P)}$ denote respectively the exact $t$-flows of the split systems corresponding to the unperturbed  dynamics
\begin{equation}\label{eq:unpertur}
\frac{d}{dt} \left[ \begin{matrix}y\\ \theta\end{matrix}\right]
= \left[ \begin{matrix}0\\ \omega\end{matrix}\right],
\end{equation}
and the perturbation
\begin{equation}\label{eq:pertur}
\frac{d}{dt} \left[ \begin{matrix}y\\ \theta\end{matrix}\right]
=
f(y,\theta).
\end{equation}
If we set
$$
a=\sum_{j=1}^ra_j,\qquad b=\sum_{j=1}^rb_j,
$$
the integrator is {\em consistent} if $a=b=1$.

Since the unperturbed dynamics with frequencies $\omega_j$ is reproduced exactly by (\ref{eq:integrat}), one would naively hope that the accuracy of the integrator would be dictated for the size of $f$ uniformly in $\omega$. It is well known that such an expectation is unjustified, see e.g. \cite{molly1}, \cite{molly2}.

\subsection{Extended word series expansion of the local error}

Clearly, the mapping $\phi_t^{(U)}$ has an expansion in extended word series
$$
\phi_t^{(U)}(x) = \overline{W}_{(t\omega,\uno)}(x),\qquad (t\omega,\uno)\in\overline{\G};
$$
furthermore,  using Example 2 in Section~\ref{sec:words},
$$
\phi_t^{(P)}(x) = \overline{W}_{(0,\tau(t))}(x),\qquad (0,\tau(t))\in\overline{\G},
$$
where $\tau(t)\in\G$  comprises the  Taylor coefficients, i.e.\ $\tau_w(t)= t^n/n!$ if $w\in\W_n$.
The following result makes use of the algebraic formalism to provide explicitly the expansion of the numerical solution.

\begin{theorem}\label{th:integrat}The splitting integrator  $\widetilde{\phi}_h$ in (\ref{eq:integrat}) possesses the expansion
$$\widetilde{\phi}_h(x) = \overline{W}_{(ha\omega,\widetilde{\alpha}(h))}(x),
$$
where
$\widetilde{\alpha}(h) \in\G$ is specified by
$\widetilde{\alpha}_\emptyset(h)=1$ and, for $n= 1,2,\dots$,
\begin{equation}\label{eq:alfatilde}\widetilde{\alpha}_{\bk_1 \cdots \bk_n}(h) = h^n\sum_{1\leq j_1 \leq \cdots \leq j_n \leq r} \frac{b_{j_1}\cdots b_{j_n}}{\sigma_{j_1 \cdots j_n}}\,
\exp(i \, (c_{j_1}\bk_1 +\cdots +c_{j_n}\bk_n  ) \cdot \omega h).
\end{equation}
Here,
 $$c_j = a_1 + \cdots + a_{j}, \qquad 1\leq j\leq r,$$
and,
\begin{eqnarray*}
  \sigma_{j_1\cdots j_n}=\frac{1}{n!} &\mbox{\rm if }& j_1 = \cdots = j_n,\\
  \sigma_{j_1\cdots j_n}=\frac{1}{\ell!}\, \sigma_{j_{\ell+1}\cdots j_n} &\mbox{\rm if }& \ell<n,\quad j_1=\cdots = j_{\ell} < j_{\ell+1} \leq \cdots \leq j_n.
\end{eqnarray*}
\end{theorem}
\begin{proof} From (\ref{eq:act2}), we know that $\widetilde{\phi}_h$ has an expansion in extended word series and that the family of coefficients is given by
(pay attention to the ordering)
$$
\big(a_1h\omega,\uno\big)\bigstar\big(0,\tau(b_1h)\big)\bigstar\cdots
\bigstar\big(a_rh\omega,\uno\big)\bigstar\big(0,\tau(b_rh)\big);
$$
it is  enough to compute, according to the definition, the products $\bigstar$ in this expression.
\end{proof}
\begin{remark}\label{rem:quad} The associated quadrature rule. \em For words with one letter, the theorem yields:
$$
\widetilde{\alpha}_{\bk}(h) = h\sum_{1\leq j \leq r} b_{j}\,
\exp(i \, c_j\bk \cdot \omega h).
$$
This obviously corresponds to the approximation of the exact coefficient $\alpha_\bk(h)$ (defined in (\ref{eq:alfa1}) and (\ref{eq:lambda})) by the (univariate) quadrature rule that on the unit interval has abscissas $c_j$ and weights $b_j$. This   rule will be consistent if $b=1$, which is implied by the consistency of the integrator.
\end{remark}
\begin{remark} Associated cubature rules. \label{rem:cub}\em Similarly, for $n>1$,
(\ref{eq:alfatilde}) corresponds to approximating (\ref{eq:alpha}) with a cubature rule for the simplex. If the univariate quadrature is consistent,  so is the cubature rule for each $n> 1$, because, by using the multinomial expansion,
\begin{eqnarray*}
&&\sum_{1\leq j_1 \leq \cdots \leq j_n \leq r} \frac{b_{j_1}\cdots b_{j_n}}{\sigma_{j_1 \cdots j_n}}=
\sum_{n_1+\cdots +n_r = n} \frac{b_1^{n_1}\cdots b_{r}^{n_r}}{n_1!\cdots n_r!} \\&&\qquad\qquad\qquad\qquad\qquad\qquad\qquad= \frac{1}{n!}
\big(\sum_{j=1}^r b_j\big)^n = \big(\sum_{j=1}^r b_j\big)^n \:{\rm Vol}({\mathcal S}(1)).
\end{eqnarray*}
\end{remark}

Discussions perhaps become clearer by introducing  scaled coefficients $A_w$ and $\widetilde{A}_w$ such that
$$\alpha_{\bk_1\cdots\bk_n}(h)= h^{n}A_{\bk_1\cdots\bk_n}(h),\qquad  \widetilde{\alpha}_{\bk_1\cdots\bk_n}(h) = h^{n}\widetilde{A}_{\bk_1\cdots\bk_n}(h) .
$$ Note that, by performing the change of variables $t_j = h t_j^\prime$, $j=1,\dots,n$, in (\ref{eq:alpha}),
\begin{equation}\label{eq:A}
A_{\bk_1\cdots\bk_n}\!(h) =\int\cdots \int_{{\mathcal S}_n(1)} \exp\big(i (t^\prime_1\bk_1+\cdots+t^\prime_n\bk_n)\cdot \omega h\big)\, dt_1^\prime\cdots dt_n^\prime,
\end{equation}
and that, therefore,
$$
\mid A_{\bk_1\cdots\bk_n}\!(h)\mid \leq {\rm Vol}({\mathcal S}_n(1)) =\frac{1}{n!}.
$$

With these preparations, we have proved our next result:
\begin{theorem} The local error of the splitting integrator  $\widetilde{\phi}_h$ in (\ref{eq:integrat}) possesses the expansion
$$
\widetilde{\phi}_h(x)-\phi_h(x) = \overline{W}_{(h(a-1)\omega,\widetilde{\alpha}(h)-\alpha(h))}(x).
$$
i.e.\
\begin{eqnarray}
&&\widetilde{\phi}_h(x_0)-\phi_h(x_0) = \left[ \begin{matrix}0\\h \big(a-1\big)\omega \end{matrix}\right] \label{eq:localformal}\\
&&\qquad \qquad + \sum_{n=1}^{\infty} h^n
\sum_{\bk_1,\dots, \bk_n\in\Z^d} \big(\widetilde{A}_{\bk_1\cdots \bk_n}(h)-A_{\bk_1\cdots \bk_n}(h)\big)\, f_{\bk_1\cdots \bk_n}(x_0).
\nonumber
\end{eqnarray}
\end{theorem}

\subsection{Estimates}

In order to obtain error estimates it is now necessary to truncate the infinite series in (\ref{eq:localformal}) and we shall do so in the next theorem, whose proof is given in Section~\ref{ss:proofestimate}. We assume hereafter that:

\begin{enumerate}
\item The function $f(x) = f(y,\theta)$ is defined in a set $\Omega = B_R(y_0)\times \mathbb{T}^d$, where $B_R(y_0)$ is the ball $\{y: |y-y_0|<R\}\subset \R^{D-d}$.
\item There exists a finite set of indices $\I\subset \Z^d$ such that for $\bk\notin \I$ the Fourier coefficient $\hat f_\bk$ vanishes.
\item There exists an integer $N\geq 2$, such that the Fourier coefficients $\hat f_\bk$ and their partial derivatives of order $\leq N-1$ are continuous and bounded in $B_R(y_0)$.
\end{enumerate}

In the first of these hypotheses the form of the domain $\Omega$ is natural since $f(y,\theta)$ is periodic in each of the components of $\theta$. As shown in e.g.\ \cite{orlando}, the second hypothesis may be relaxed for the conclusions of the theorem to hold;   it  makes however possible to avoid distracting technicalities. In this connection it should be noted that, for nonlinear problems, even if $f$  has a finite number of Fourier modes  the solution $x(t)$ in (\ref{eq:solucion}) will include arbitrarily high frequencies (the product of $\lambda$'s in (\ref{eq:alpha}) adds the corresponding wave numbers $\bk$).\footnote{For smooth solutions, terms with high frequency must have small amplitude, a fact that may be exploited in the derivation of error bounds \cite{molly1}, \cite{molly2}. This point will not be studied here.}
\begin{theorem}
\label{th:estim}
Assume  that the system (\ref{eq:ode2})  being integrated satisfies the assumptions above.
Then there exist positive constants $h_0$, $C$, {\em both independent of $\omega$,} such that:
\begin{enumerate}
\item For $|h|<h_0$ and arbitrary $\theta_0$, the true solution $\phi_h(x_0)$, $x_0=(y_0,\theta_0)$, and the numerical solution $\widetilde{\phi}_h(x_0)$  are well defined and lie in $\Omega$.
\item The local error at $x_0$ satisfies
\begin{eqnarray}
&&\widetilde{\phi}_h(x_0)-\phi_h(x_0) = \left[ \begin{matrix}0\\h \big(a-1\big)\omega \end{matrix}\right] \nonumber\\
&&\qquad \qquad + \sum_{n=1}^{N-1} h^n
\sum_{\bk_1,\dots, \bk_n\in\I} \big(\widetilde{A}_{\bk_1\cdots \bk_n}(h)-A_{\bk_1\cdots \bk_n}(h)\big)\, f_{\bk_1\cdots \bk_n}(x_0)\nonumber\\
&&\qquad\qquad +{\calR}_h(x_0),\label{eq:estimate}
\end{eqnarray}
where $|\calR_h(x_0)|\leq C|h|^N$.
\end{enumerate}
\end{theorem}

The theorem reduces the estimation of the local error to the estimation of the quantities $\widetilde{A}_{\bk_1\cdots \bk_n}(h)-A_{\bk_1\cdots \bk_n}(h)$. These are errors arising in the quadrature of scalar smooth trigonometric functions and {\em  are completely independent of the function $f$}.

{\em It is assumed hereafter that the integrator is  consistent}. We first analyze the local error in the limit $h\rightarrow 0$. The condition $a=1$ implies that the first term in the right-hand side of (\ref{eq:estimate}) vanishes. Furthermore, from Remark~\ref{rem:quad}, $\widetilde{A}_{\bk}(h)-A_{\bk}(h)=\mathcal{O}(h)$ as $h\rightarrow 0$ and we conclude that $\widetilde{\phi}_h(x)-\phi_h(x)=\mathcal{O}(h^2)$.
Note that, for the word with ${\bf 0}\in\Z^d$ as its only letter, $\widetilde{A}_{\bf 0}(h)-A_{\bf 0}(h)=0$.
Moreover, in view of Remark~\ref{rem:cub}, for each $n\leq N-1$, the $n$-th term in the sum in (\ref{eq:estimate}) is actually $\mathcal{O}(h^{n+1})$ rather than
$\mathcal{O}(h^{n})$.

 If, additionally, the underlying univariate quadrature rule is second-order accurate, i.e.\ $\sum b_jc_j = 1/2$, then $\widetilde{A}_{\bk}(h)-A_{\bk}(h)=\mathcal{O}(h^2)$, and the integrator will be second order accurate, $\widetilde{\phi}_h(x)-\phi_h(x)=\mathcal{O}(h^3)$, provided that  hypothesis 3 holds with  $N\geq 3$.

The argument may be taken further to translate accuracy properties of the associated quadrature and cubature rules into accuracy properties of the integrator in the limit $h\rightarrow 0$. In this way one recovers the order conditions for splitting methods  listed in \cite{new} (cf. \cite{royal}). We shall not pursue that path:  our interest lies in the size of the local error when $h$ is not small relative to the periods present in the  dynamics, a scenario that we discuss next.

It is well known that for a quadrature rule that is exact for  polynomials of degree $\leq \sigma$,
$$|\widetilde{A}_{\bk}(h)-A_{\bk}(h)| \leq C |\bk \cdot \omega|^{\sigma+1} h^{\sigma+1},$$
for a constant $C$ that only depends on the rule. Therefore for the quadrature errors to be small it is {\em  necessary that $|h|$ be small with respect to $\min (1/|\bk \cdot \omega|)$}, where the minimum is extended to all $\bk$ with
$\bk\cdot\omega\neq 0$. We reach the unwelcome conclusion that the size of the bound in (\ref{eq:estimate}) depends {\em both} on the size of the perturbation $f$  {\em and} on  $\omega$.

 \begin{example}\em Consider the familiar Strang splitting, $r=2$,
 \begin{equation}\label{eq:strang}
 a_1=1/2, \quad a_2 =1/2,\quad b_1=1, \quad b_2=0.
 \end{equation}
 The underlying quadrature formula is the second-order accurate midpoint rule. For this integrator, for each $\bk$ such that $\bk\cdot\omega\neq 0$,
\begin{equation}\label{eq:strangquad}
\widetilde{A}_{\bk}(h)-A_{\bk}(h) =
\exp((1/2)i\bk \cdot \omega h) - \frac{\exp(i\bk \cdot \omega h)-1}{i\bk \cdot \omega h}
\end{equation}
(for $\bk\cdot\omega=0$, $\widetilde{A}_{\bk}(h)=A_{\bk}(h)=1$).
An elementary computation leads to the bound
$$
|\widetilde{A}_{\bk}(h)-A_{\bk}(h)|\leq \frac{1}{24} |\bk \cdot \omega|^2 h^2,
$$
where the constant $1/24$ cannot be improved if the inequality has to hold for arbitrary  $h$. $\Box$
\end{example}

\begin{remark}\label{rem:counter} \em The dependence on $\omega$ of the local error is not an artifact introduced by our method of analysis.
 Here is an example. Consider the forced spring (see the Appendix), $(d/dt) p =-\omega^2 q+F$, $dq/dt = p$, where $F\neq 0$ is a time-independent force and $\omega> 0$. This is the Hamiltonian system with Hamiltonian $H = (1/2)p^2+(\omega^2/2)q^2-qF$ or, in action-angle variables
 $$H = \omega a-\sqrt{\frac{2a}{\omega}}\sin \theta\: F = \omega a -\frac{1}{2i} \sqrt{\frac{2a}{\omega}}\exp(i\theta) F+ \frac{1}{2i} \sqrt{\frac{2a}{\omega}}\exp(-i\theta) F.
 $$
There are two Fourier modes $k=\pm 1$ in the perturbation.

Choose initial conditions $p_0=1$, $q_0=0$ (with kinetic energy $1/2$ and no potential energy in the spring). If $h/(1/\omega)= 2\pi$, after one time step, the true solution has $p(h) = 1$ and Strang's method (\ref{eq:strang}) yields an approximation $\widetilde p(h) =
1-hF$; therefore a bound of the form $|\widetilde{p}(h)-p(h)|\leq C|h|^{\sigma+1}$, $|h|<h_0$, with $C$ and $h_0$ independent of $\omega$ cannot exist for $\sigma>0$.  Note that, after $m$ steps, the error in $p$ will be $mh$!
\end{remark}
\begin{figure}[t]
\begin{center}
%\vskip -7 truecm
\includegraphics[scale=.80]{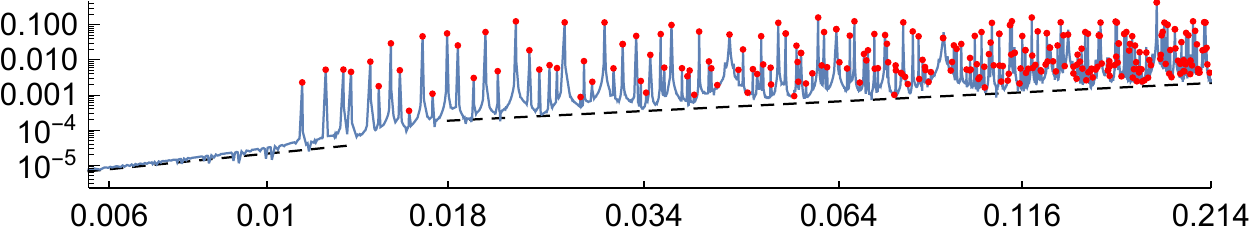}% MATLAB hmmresorte
\end{center}
%\vskip -4 truecm
\caption{Energy error at  time $T=50$ as a function of $h$ (in doubly logarithmic scale) for Example~\ref{example:fpu}. The discontinuous straight lines correspond to $\mathcal{O}(h^2)$ (left) and $\mathcal{O}(h)$ (right). Small circles have been located at points whose abscissa is a value of $h$ that leads to a first order numerical resonance (Section~\ref{ss:processing}). It is apparent that those points give rise to local maxima of the error.
}
\label{fig:picos}
\end{figure}

\begin{example}
\label{example:fpu}
\em In order to observe the behavior of the Strang's method (\ref{eq:strang}) in problems more involved than the scalar example in the last remark, we have integrated the Hamiltonian problem with $d=5$ degrees of freedom with Hamiltonian function from \cite[Chapter XIII.9]{hlw}
$$
\frac{1}{2} \sum_{j=1}^5 \big((p^j)^2+ {\omega_j}^2 (q^j)^2\big) + U(q),
$$
with
$$
U(q) = \frac{1}{8} (q^1q^2)^2+\left(\frac{1}{20} +q^2+q^3+q^4+\frac{5}{2}q^5\right)^4
$$
and $\omega_1 = 1$, $\omega_2=\omega_3=70$, $\omega_4 = 70\sqrt{2}$, $\omega_5=2\omega_2$.
The system is split by dividing the Hamiltonian into its harmonic (quadratic) part, with linear dynamics, and the perturbation corresponding to $U$ (see the Appendix).  We integrated this problem over the (very long) interval $0\leq t\leq 35000$ with the initial condition (also taken from \cite{hlw})
$$
p_0 = \left(-\frac{1}{5}, \frac{3}{5}, \frac{7}{10}, -\frac{9}{10}, \frac{4}{5}\right),\qquad
 q_0 = \left(1, \frac{3}{10\omega_2}, \frac{4}{5\omega_2},- \frac{11\sqrt{2}}{10\omega_4}, \frac{7}{10\omega_2}\right),
$$
for which the energy in the harmonic part is $4.225$. Figure~\ref{fig:picos} gives the error in the Hamiltonian at  time $t=50$ as a function of $h$.  Two different regimes are apparent in the figure:

\begin{enumerate}
\item For $h$ small, the error in the Hamiltonian is very approximately $ Ch^2$, as it corresponds to the second order of accuracy of Strang's splitting. In this regime the oscillatory nature of the problem is not relevant and the integrator may be analyzed by standard techniques, i.e.\ expansion of the local error {\em in powers of $h$} and transference, using stability, of local error bounds to bounds of the global error.
\item For $h$ large, the  error presents a very irregular behavior. This is due to the highly oscillatory character of the solution and, as we shall now describe, may be analyzed via the {\em word series expansion} of the local error.
\end{enumerate}

\end{example}

\subsection{Processing}
\label{ss:processing}
 It is clear that, as distinct from the global error, the quadrature error in (\ref{eq:strangquad}) varies regularly as $h$ varies. The irregularities in Fig.~\ref{fig:picos} stem from  cancelations, due to the oscillations, of local errors in consecutive time steps. For this reason, sharp  error estimates in highly oscillatory problems (see e.g.~\cite{molly1}, \cite{molly2}) do not bound the local error and then sum the bounds; they rather sum first and bound later, so as to take advantage of possible cancelations.  We use here an alternative approach that exploits the idea of {\em processing} that goes back to Butcher \cite{butcher69}. The presentation here follows \cite{cheap}.

If $\chi_h$ is  a near-identity mapping in $\R^D$ and $\widetilde{\phi}_h$ is an integrator, the mapping
\begin{equation}\label{eq:conj}
\widehat{\phi}_h = \chi_h^{-1}\circ \widetilde{\phi}_h\circ \chi_h
\end{equation}
defines a {\em processed} numerical integrator. For $m\geq 1$
\begin{equation}\label{eq:ksteps}
\widehat{\phi}_h^m = \big(\chi_h^{-1}\circ \widetilde{\phi}_h\circ \chi_h\big)^m = \chi_h^{-1}\circ \widetilde{\phi}_h^m\circ \chi_h;
\end{equation}
therefore to advance $m$ steps with  $\widehat{\phi}_h$ one may preprocess the initial condition to find $\chi_h(x_0)$, advance $m$ steps with the original method and then postprocess the numerical solution by applying $\chi^{-1}_h$. Postprocessing is only performed when output is desired, not at every time step. In practice, the idea of processing is useful if $\chi_h$ may be chosen in such a way that $\widehat{\phi}_h$ is more accurate in some sense than the original $\widetilde{\phi}_h$: one then obtains extra accuracy at the (hopefully small) price of having to perform the processing (this gives rise to Butcher's notion of effective order \cite{butcher69}, \cite{bss}). Here we use the idea of processing as a {\em technique of analysis}. We shall  process the splitting  method (\ref{eq:integrat}) by means of a mapping $\chi_h$ with an expansion in  word series: $\chi_h(x) = W_{\kappa(h)}(x)$, $\kappa(h)\in\G$ (see Remark~\ref{rem:generalized}). Then the processed integrator will  be expressible as an extended word series with coefficients of the form $(h\omega,\widehat{\alpha}(h))\in\overline{\G}$. By implication, the local error will be of the form (\ref{eq:localformal}) with $\widehat{A}_{\bk_1\cdots \bk_n}(h)-A_{\bk_1\cdots \bk_n}(h)$ in lieu of $ \widetilde{A}_{\bk_1\cdots \bk_n}(h)-A_{\bk_1\cdots \bk_n}(h)$. (We have used the obvious notation $h^n\widehat{A}_{\bk_1\cdots \bk_n}(h) =\widehat{\alpha}_{\bk_1\cdots \bk_n}$ for the scaled coefficients of the processed method and will similarly set
$h^nK_{\bk_1\cdots \bk_n}(h) =\kappa_{\bk_1\cdots \bk_n}(h)$.) Our policy is to determine the processing, i.e.\ to determine
$\kappa(h)\in\G$, in such a way that $\widehat{A}_{\bk_1\cdots \bk_n}(h)-A_{\bk_1\cdots \bk_n}(h)=0$, for all {\em oscillatory words}. We then may hope that the analysis of the processed integrator would be free from the difficulties usually associated with integrators of oscillatory problems. Finally the results on the processed integrator obtained in this way  will be translated into results for the original $\widetilde{\phi}_h$.

\subsubsection{First-order numerical resonances}

The conjugation  (\ref{eq:conj}) of the mappings may be  translated with the help of (\ref{eq:act2}) into the equation
\begin{equation}\label{eq:con2}
(h\omega,\widehat{\alpha}(h)) \bigstar (0,\kappa(h))= (0,\kappa(h))\bigstar (h\omega,\widetilde \alpha(h))
\end{equation}
for the coefficients.
 For words with one letter, (\ref{eq:con2}) implies, according to the definition of $\bigstar$:
$$
\exp(i\bk\cdot\omega h)\:K_\bk(h)+ \widehat{A}_\bk(h) = \widetilde{A}_\bk(h)+K_\bk(h).
$$
There are two cases to be analyzed. We first look  at words (including $\bk={\bf 0}$) that are not oscillatory, i.e.\
$\bk\cdot\omega = 0$.
The value
$K_{\bf k}(h)$ drops from (\ref{eq:con2}) and  may be regarded as a free parameter.
In addition, $\widehat{A}_{\bf k}(h)=\widetilde{A}_{\bf k}(h)= 1$ and therefore  $\widehat{A}_{\bf k}(h)-A_{\bf k}(h)= 0$. We then consider  oscillatory one-letter words $\bk$,
$\bk\cdot\omega \neq 0$,  and according to our policy, we try to get $\widehat{A}_\bk(h)=A_\bk(h)$. This  leads to
 \begin{equation}\label{eq:blow}
 K_\bk(h) = \frac{\widetilde{A}_\bk(h)-{A}_\bk(h)}{\exp(i\bk\cdot\omega h)-1},
 \end{equation}
 {\em provided that} $\exp(i\bk\cdot\omega h)\neq 1$.
  If $\bk\cdot\omega \neq 0$ and $\exp(i\bk\cdot\omega h)= 1$, we say that a {\em first-order numerical resonance} occurs. When this happens,  $K_\bk(h)$ drops from  (\ref{eq:con2}) and $\widehat{A}_{\bk}(h)=\widetilde{A}_{\bk}(h)$. As a consequence, in general,  $\widehat{A}_{\bk}(h)-{A}_{\bk}(h)$ will not vanish.

If, for given $h$,  there is no first-order numerical resonance, then the expansion of the local error only contains terms corresponding to words with two or more letters. In analogy with Theorem~\ref{th:estim} (details will not be given), it is then possible to bound the local error of the processed integrator by $Ch^2$, with $C$ independent of $\omega$. This in turn will lead to a $C^\prime h$ bound for the global error of the processed integrator and, after taking into account the pre- and postprocessing to a $C^{\prime\prime}h$ bound for the global error in the method $\widetilde{\phi}_h$ being analyzed. The constant $C^{\prime\prime}$ may be chosen to be independent of $h$, provided that $h$ is bounded away from the resonances; it worsens as $h$ gets closer to a numerical resonance in view of (\ref{eq:blow}). This explains the troughs in Fig.~\ref{fig:picos}.

On the other hand, if for given $h$ there is at least one numerically resonant $\bk\in\I$, then, processing is of no help in removing the $\omega$-dependent quadrature error of the original, unprocessed method.
This was only to be expected because at a numerical resonance,
as shown in Remark~\ref{rem:counter}, the global error in $\widetilde{\phi}_h$ may actually be large (cusps in Fig.~\ref{fig:picos}).

\begin{remark} \em By using the operation $\bigstar$ to compute the expansion of the $m$-fold compositions $\widetilde\phi_h^m$ and $\phi_h^m$, we find after some simple algebra that
if $\pm\bl$  are numerically resonant wavenumbers and $\exp(i\bk\cdot\omega h)\neq 1$ for $\bk\neq {\bf 0}, \pm \bl$, then the error over $m$ steps has an expansion
\begin{eqnarray*}
&& \widetilde\phi_h^m(x_0)-\phi_h^m(x_0) = \\ && \qquad mh \:\Big( \widetilde A_\bl(h) f_\bl(x_0)+\widetilde A_{-\bl}(h) f_{-\bl}(x_0\Big) +\\&&\qquad
 h\sum_{\bk\in \I \backslash \{{\bf 0},\pm \bl \}}
\frac{\exp(i\bk\cdot\omega mh)-1}{\exp(i\bk\cdot\omega h)-1}\:
 \Big(\widetilde A_\bk(h)-A_\bk(h)\Big) f_\bk(x_0)+\cdots
\end{eqnarray*}
 Thus  the $mh$ growth as $m$ increases with fixed $h$ we already encountered in Remark~\ref{rem:counter} holds for general integrators and general differential equations.
\end{remark}

\subsubsection{Higher-order resonances}
Assuming that $h$ does not satisfy any first-order numerical resonance, one may go a step further and look at words with two letters $\bk\bl$; these are oscillatory if $(\bk+\bl)\cdot\omega \neq 0$. Now (\ref{eq:con2}) implies
\begin{eqnarray*}
\exp(i(\bk+\bl)\cdot\omega h) \:K_{\bk\bl}(h) + \widehat{A}_\bk(h) \exp(i\bl\cdot \omega h) K_\bl(h) +\widehat{A}_{\bk\bl}(h)\qquad\qquad{} \\
=\widetilde{A}_{\bk\bl}(h)+K_{\bk}(h)\widetilde{A}_{\bl}(h)+K_{\bk\bl}(h).
\end{eqnarray*}
Whenever
$\exp(i(\bk+\bl)\cdot\omega h) = 1$ (second order numerical resonance), the value of $K_{\bk\bl}(h)$ cannot be chosen to ensure that $\widehat{A}_{\bk\bl}(h)={A}_{\bk\bl}(h)$. A similar consideration applies for nonoscillatory words with $n$ letters when $\exp(i(\bk_1+\cdots +\bk_n)\cdot\omega h) = 1$. When no numerical resonance takes place, the equation (\ref{eq:con2}) may be used to find  values  $\kappa_w(h)$, $w\in\W$ such that, on the one hand, define an element $\kappa(h)$ that  belongs to $\G$, (i.e.\ the shuffle relations hold) and, on the other, ensure that $\widehat{A}_{\bk_1\cdots \bk_n}(h)-A_{\bk_1\cdots \bk_n}(h)=0$, for all {\em oscillatory words}. This will be proved in Remark~\ref{rem:final} below by using modified systems.

\begin{remark}\em Since  pre- and postprocessing introduce in any case $\mathcal{O}(h)$ errors, the processing technique used here yields  $\mathcal{O}(h)$  bounds for the global error of $\widetilde{\phi}_h$ even for values of $h$ where there is no first-order or second-order numerical resonances. Fig.~\ref{fig:picos} shows that, for this simulation, $\mathcal{O}(h^2)$ global error bounds cannot exist if $h$ is large relative to the periods present in the dynamics.
\end{remark}

\subsection{Modified equations and modified Hamiltonians}
Modified equations \cite{gss}, \cite{aust}, \cite{ssc}, \cite{hlw} provide a useful means to describe the behaviour of numerical integrators.

\subsubsection{Modified system using one letter words}

  We  look for a  (one-letter word) modified system
\begin{equation}\label{eq:mod}
\frac{d}{dt} \widetilde x = \overline{W}_{(\omega,\widetilde\beta(h))}(\widetilde x),
\end{equation}
where $\widetilde\beta\in \g$, $\widetilde\beta_w = 0$ for $w\in \W_n$, $n>1$ and the coefficients $\widetilde\beta_\bk(h)$, $\bk \in\I$ are chosen in such a way that, for words with one letter, the extended word series expansion of the  $h$-flow $\widetilde\phi_h^{[1]}$ of (\ref{eq:mod}) matches the corresponding expansion for the integrator $\widetilde \phi_h$. We recall that the system being solved is also of the form (\ref{eq:mod}) with the coefficients $\beta$ given in (\ref{eq:b}).

By integrating the system (\ref{eq:mod}) by the procedure outlined in Section~\ref{sec:extended} and imposing that its flow matches $\widetilde \phi_h$ to the desired order, we find the condition
\begin{equation}\label{eq:betabeta}
 \frac{\exp(i\bk\cdot\omega h)-1}{i\bk\cdot\omega h}\: \widetilde\beta_\bk(h) = \widetilde A_\bk(h)
\end{equation}
(it is understood that the fraction takes the value $1$ if $\bk\cdot\omega =0$).
For $\bk = \bf 0$ or for any one letter word that is not oscillatory, this implies $\widetilde\beta_\bk(h)=1$. For an  oscillatory one-letter word $\bk \neq \bf 0$, $\bk\in\I$, if $h$ is such that $i\bk\cdot \omega h=2\pi j$ for some integer $j\neq 0$ (first order numerical resonance), then the fraction in (\ref{eq:betabeta}) vanishes and the equation for
$\widetilde\beta_\bk(h)$ will in general not be solvable. Thus first-order numerical resonances are obstructions to the construction of the modified system.  (In fact the counterexample in Remark \ref{rem:counter} proves that modified systems of the form envisaged here do not exist at numerical resonances.)
When $h$ is bounded away from resonances,
$\widetilde\phi_h^{[1]}-\widetilde\phi_h=\mathcal{O}(h^2)$ with the implied constant independent of $\omega$ (as in Theorem~\ref{th:estim}).

\begin{example}\em For Strang's method,  if $\bk$ is oscillatory and there is not a numerical resonance, (\ref{eq:betabeta}) yields the value
$$\widetilde\beta_\bk(h)= \frac{\bk\cdot\omega h}{2\sin (\bk\cdot\omega h/2)}.\qquad \Box $$
\end{example}

\begin{remark}\label{rem:modified} \em Assume that, for given $h$, the modified system above has been found. We may then try to find a change of variables $\widetilde x = W_{\kappa(h)}(\widetilde X)$ so that in the new variables the modified vector field matches
the field $\overline{W}_{(\omega,\beta)}(X)$ for words with one letter. According to Section~\ref{sec:extended}, we have to impose that
$
\beta\star \kappa(h) + \xi_{\omega} \kappa(h)$ and  $\kappa(h)\bigstar \widetilde \beta(h)$
coincide for words with one letter. This leads to
$$i(\bk\cdot \omega)\:\kappa_\bk(h) +1= \widetilde \beta_\bk(h).$$ If $\bk$ is not oscillatory, $\kappa_\bk(h)$ is free because, as noted above, $\widetilde \beta_\bk(h) =1$.
For $\bk$  oscillatory,  $\kappa_\bk(h)$ is uniquely determined. By using
(\ref{eq:betabeta}),  a little algebra shows that the value of $\kappa_\bk(h)$ found in this way is the same we obtained in (\ref{eq:blow}). Thus the change of variables $W_{\kappa(h)}$ we used for processing may be seen as determined by the requirement that, in the new variables and for nonoscillatory one-letter words, the modified vector field of the unprocessed integrator reproduces the vector field being integrated.
\end{remark}

\subsubsection{Other modified systems}

More precise modified systems may be constructed by successively adding to the modified vector field contributions from words of 2, 3, \dots\ letters. For the $n$-th of these modified systems,
the modified vector field has $\widetilde\beta_w = 0$ for words with more than $n$ letters and we impose that,
for words with $n$ or fewer letters,
the extended word series expansion of the  $h$-flow $\widetilde\phi_h^{[n]}$  matches the corresponding expansion for the integrator $\widetilde \phi_h$.

For two-letter words, proceeding as in the case of one-letter word modified systems, we obtain the condition
$$
 \frac{\exp(i(\bk+\bl)\cdot\omega h)-1}{i(\bk+\bl)\cdot\omega h}\:\widetilde\beta_{\bk \bl}(h) + h A_{\bk\bl}(h) \widetilde{\beta}_{\bk}(h) \widetilde{\beta}_{\bl}(h)= h \widetilde A_{\bk\bl}(h) .
$$
%$$
% A_{\bk+\bl}(h)\: \widetilde\beta_{\bk \bl}(h) + h A_{\bk\bl}(h) \widetilde{\beta}_{\bk}(h) \widetilde{\beta}_{\bl}(h)= h \widetilde A_{\bk\bl}(h) .
%$$
%
In general, the resulting  equation is of the form
$$\frac{\exp(i(\bk_1+\cdots +\bk_n)\cdot\omega h)-1}{i(\bk_1+\cdots +\bk_n)\cdot\omega h}\: \widetilde\beta_{\bk_1 \cdots \bk_n}(h)-h^{n-1}  \widetilde A_{\bk_1\cdots \bk_n}(h)=\mathcal{O}(h^{n-1}),$$
%$$A_{\bk_1+\cdots +\bk_n}(h)\: \widetilde\beta_{\bk_1 \cdots \bk_n}(h)-h^{n-1}  \widetilde A_{\bk_1\cdots \bk_n}(h)=\mathcal{O}(h^{n-1}),$$
where  the right-hand side depends polynomially on the coefficients $\widetilde{\beta}_{w}(h)$ of words $w$ with less than $n$ letters.
Such an equation can be solved for $\widetilde\beta_{\bk_1 \cdots \bk_n}(h)$
provided that there is no numerical resonance, $(\bk_1+\cdots +\bk_r)\cdot \omega h=2\pi j$, $r\leq n$,  $j\neq 0$. In the limit where the length of the words increases indefinitely one obtains, if there is no numerical resonance of any order, a modified system whose formal $h$-flow exactly reproduces the expansion of $\widetilde \phi_h$.

As explained in Section~\ref{sec:words} the modified systems found in this way are Hamiltonian whenever the system being integrated is Hamiltonian. Furthermore the modified Hamiltonian functions are easily expressible in terms of brackets.

\begin{example}\em For the linear forced oscillator in Remark~\ref{rem:counter}, each word basis functions associated with words with two or more letters vanishes. In this case  the modified systems above with $n>1$, coincide with the modified system using only one-letter words and the latter is exact, ie $\widetilde\phi_h^{[1]}= \widetilde \phi_h$. The (exact) modified Hamiltonian is
$$
\frac{1}{2}p^2 +\frac{\omega^2}{2} q^2 - q\, \widetilde\beta_1(h)\, F,
$$
and therefore in the particular case of Strang's method we have
$$
\frac{1}{2}p^2 +\frac{\omega^2}{2} q^2 - q\, \frac{\omega h}{2\sin (\omega h/2)}\, F.
$$
For nonresonant $h$, the effect of using the splitting method is to alter the value of the applied force. Unless $|\omega h|\ll 1$ the misrepresentation of the force introduced by the discretisation will be large. We emphasize that, as distinct from the situation when using conventional modified equations based on series of powers of $h$, the analysis here does not require $h$ to be small. $\Box$
\end{example}

\begin{figure}[p]
\begin{center}
%\vskip -7 truecm
\includegraphics[scale=.80]{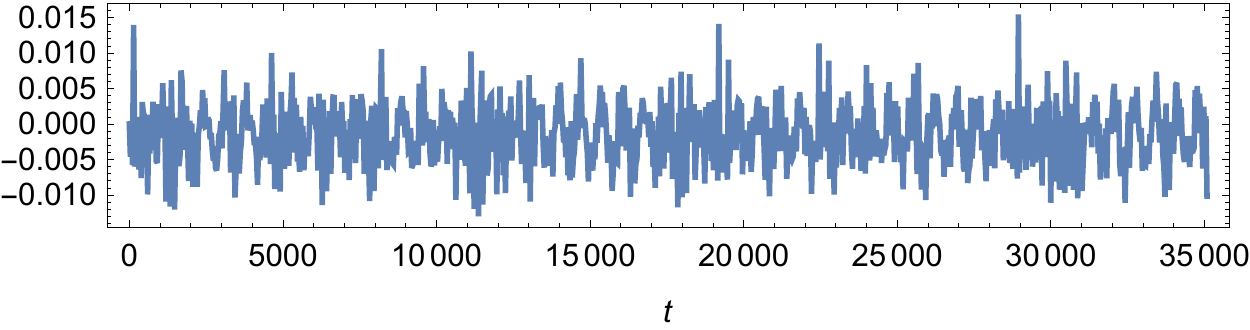}% MATLAB hmmresorte
\end{center}
%\vskip -4 truecm
\caption{Variation of the true Hamiltonian (energy) evaluated at the numerical solution as a function of $t$ for Example~\ref{example:fpu} ($h= 0.7974$).}
\label{fig:mod0}
\end{figure}
\begin{figure}[p]
\begin{center}
%\vskip -7 truecm
\includegraphics[scale=.80]{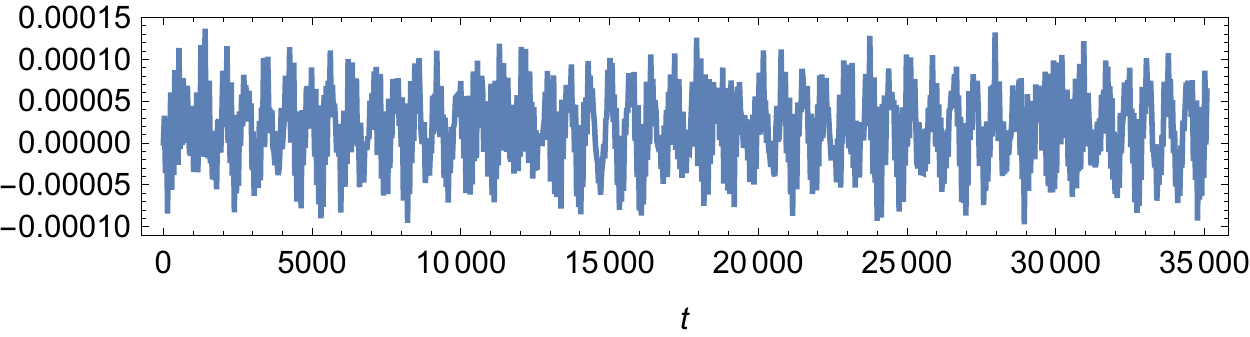}% MATLAB hmmresorte
\end{center}
%\vskip -4 truecm
\caption{Variation of the one-letter-word Hamiltonian evaluated at the numerical solution as a function of $t$ for Example~\ref{example:fpu} ($h= 0.7974$). The vertical scale is  $100$ times larger than in  Fig.~\ref{fig:mod0}.}
\label{fig:mod1}
\end{figure}
\begin{figure}[p]
\begin{center}
%\vskip -7 truecm
\includegraphics[scale=.80]{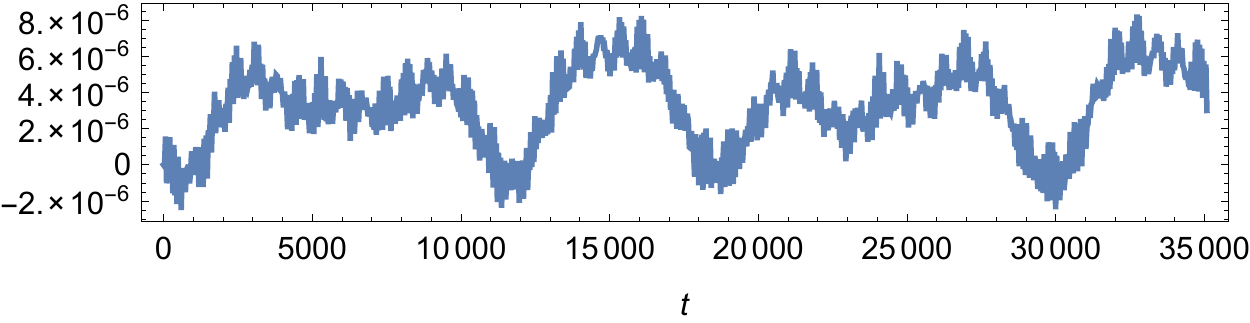}% MATLAB hmmresorte
\end{center}
%\vskip -4 truecm
\caption{Variation of  the two-letter-word Hamiltonian evaluated at the numerical solution as a function of $t$ for Example~\ref{example:fpu} ($h= 0.7974$).}
\label{fig:mod2}
\end{figure}

\begin{example} \em For the problem in Example~\ref{example:fpu}  we have measured the variation as $t$ increases of the true energy $H$ of the numerical solution and of the corresponding variations of the energies in the modified one-letter-word Hamiltonian and two-letter-word Hamiltonians. We used $h = 0.7974$ as this, which is more than 12 times larger than the period of the fastest oscillator, avoids first and second order resonances. The results given in Figs.~\ref{fig:mod0}--\ref{fig:mod1} clearly bear out how the one-word-letter modified system matches the numerical solution much better than the system being integrated, but not as well as the two-word-letter modified system.
\end{example}

\begin{remark}
\label{rem:final} \em The idea in Remark~\ref{rem:modified} may be extended. Assume that there is no numerical resonance of any order so that it is possible to find a modified equation  whose formal flow exactly reproduces the expansion of the integrator. By using Theorem~\ref{th:theoremnormal} we may bring the modified system to a normal form where the contribution of all oscillatory words have disappeared. This implies that a processor has been found such that the expansion of the local error of the processed method does not contain contributions of oscillatory words.
\end{remark}

 We close this section with an observation. As noted before, {\em numerical resonances} ($(\bk_1+\cdots +\bk_r)\cdot \omega h=2\pi j$,   $j\neq 0$) obstruct the construction of modified systems; nonoscillatory words
($(\bk_1+\cdots +\bk_r)\cdot \omega = 0$) cause no trouble in that connection. On the other hand, the nonoscillatory character of a word {\em is} an obstruction to its elimination by changing variables. Processing, that as we have just seen is equivalent to finding modified problems {\em and} then changing variables, is hampered by both numerical resonances and non-oscillatory terms. See in this connection (\ref{eq:blow}), whose denominator vanishes both at numerical resonances {\em and} for nonoscillatory words.

\section{Technical results}
\label{s:technical}

This section is devoted to more technical material.

\subsection{Algebraic results}
\label{ss:technicalwords}

\subsubsection{Differential operators}
In (\ref{eq:revision}) we associated with each vector field $f_a$ in (\ref{eq:ode})  a first-order linear differential operator $E_{a}$.
With each word $w= a_1\cdots a_n$, $n>0$ we now associate the $n$-th order (linear) differential operator $E_w$ obtained by composition:
$$
E_{a_1\cdots a_n} g(x) = E_{a_1}\cdot E_{a_2} \cdots E_{a_n} g(x);
$$
 $E_\emptyset$ is defined as the identity operator. Finally, with each $\gamma\in \C^\W$ we associate the formal series of linear differential operators
$$
D_\gamma = \sum_{w\in\W} \gamma_w E_w.
$$

Two non-empty words $w= a_1\cdots a_m$, $w^\prime = a^\prime_1\cdots a^\prime_n$ may be {\em concatenated} \cite{reu} to give rise to the word
$ww^\prime= a_1\cdots a_ma^\prime_1\cdots a^\prime_n$. In addition $\emptyset w =w\emptyset = w$ for each $w\in\W$. Clearly,  concatenation of words corresponds to composition of the associated operators:
$E_{ww^\prime} = E_{w}\cdot E_{w^\prime}$.

Each word $a_1\cdots a_n$ may  be {\em deconcatenated} in $n+1$ different ways: $\emptyset(a_1\cdots a_n)$,
$(a_1)(a_2\cdots a_n)$, \dots,
$(a_1\cdots a_n)\emptyset$; these feature in the definition (\ref{eq:convol}) of the convolution product. This observation leads to  the following rule for the composition of two series of operators:
\begin{equation}\label{eq:compD}
D_\gamma \cdot D_{\gamma^\prime} = D_{\gamma\star\gamma^\prime}.
\end{equation}

\subsubsection{The shuffle algebra}
\label{sss:shuffle}

The product $\sh$ may be extended in a bilinear way from words to linear combinations
  of words, i.e. if $\mu_j$, $\mu_{j^\prime}$ are scalars:
\[
\big(\sum_j \mu_j w_j\big) \sh \big(\sum_{j^\prime} \mu_{j^\prime} w_{j^\prime}\big) = \sum_{j,j^\prime} \mu_j\mu_{j^\prime} w_j\sh w_{j^\prime}.
\]
When endowed with this operation,
the vector space  $\C\langle A\rangle $ of all such  linear combinations   is  a unital, commutative, associative algebra, {\em the shuffle algebra}, denoted by sh$(A)$ (see \cite{reu}, \cite{fm}, \cite{anderfocm}). Note that sh$(A)$ is graded by the number of letters of the words.

Deconcatenation defines a coproduct and turns sh$(A)$ into a (commutative, connected, graded) {\em Hopf algebra} \cite{brouder}. It is well known  that the dual vector space of a Hopf algebra  is automatically endowed with a product operation. Here the dual of  sh$(A)$ may be identified in a natural way with $\C^\W$ by associating with each linear form $\ell$ on sh$(A)$  the family of coefficients $\gamma_w =\ell(w)$, $w\in\W$. After this identification, the product in the dual of sh$(A)$ coincides with the  convolution product $\star$ defined in (\ref{eq:convol}). The sets $\G$ and $\g$ in Section \ref{sec:words} are then respectively the  group of characters and the Lie algebra of infinitesimal characters of the Hopf algebra sh$(A)$; well known results on Hopf algebras show that $\exp_\star$ in (\ref{eq:exp}) maps
$\g$ onto $\G$ and has an inverse given by $\log_\star$ in (\ref{eq:log}), see e.g.\ \cite{anderfocm}, \cite{fm}.

\subsubsection{The actions of $\G$ and $\g$ on word series}
\label{sss:action}

As shown e.g.\ in \cite{orlando}, there is a narrow connection between the word basis functions $f_w(x)$ and the operators $E_w$, $w\in \W$:
$$
f_w(x) = E_w x,
$$
(in the right-hand side, with an abuse of notation, $x$ denotes the identity function that maps each $D$-vector into itself). As a consequence we have the following correspondence between word series and series of operators
\begin{equation}\label{eq:wd}
W_\gamma(x) = D_\gamma x.
\end{equation}
The use of series of {\em operators} is common in control theory and dynamical systems;  word series, being series of {\em functions}, provide a more convenient way to study numerical integrators.

The operators $E_a$, $a\in A$, are derivations: $E_a (gh) = (E_a g) h + g (E_a h)$ for each pair of  scalar functions $g$, $h$. Iteration yields:
\begin{eqnarray*}
E_{ab} (gh) &=& (E_{ab} g) (E_\emptyset h) + (E_b g)(E_a h)+ (E_a g)(E_b h)+ (E_\emptyset g) (E_{ab} h),\\
E_{abc}(gh) &=& (E_{abc} g) (E_\emptyset h) + (E_{bc} g)(E_a h)+ (E_{ac} g)(E_b h)+  (E_{c} g) (E_{ab} h)+\\
&& (E_{ab} g)(E_c h)+(E_{b} g)(E_{ac} h)+ (E_a g)(E_{bc} h)+ (E_\emptyset g) (E_{abc} h),
\end{eqnarray*}
etc.
Note that, in the first of these identities, the pairs of words $(ab,\emptyset)$, $(a, b)$, $(b, a)$, $(\emptyset, ab)$ that feature in the right-hand side are precisely those whose shuffle product gives rise to the word $ab$ that appears in the left-hand side. A similar observation may be made in the second identity.
In general, if $w\in \W_m$, $w^\prime\in \W_n$ and $w\sh w^\prime = \sum_j w_j$, then the  $w_j\in\W_{m+n}$ are precisely those words for which  $(E_w g)(E_{w^\prime}h)$ is one of the $2^{m+n}$ terms of the expansion
of $E_{w_j}(gh)$.\footnote{Algebraically, the action of the operators $E_w$ on products $gh$ defines a coproduct \cite{brouder}; the shuffle product is obtained from this coproduct by duality \cite[Section 1.5]{reu}.} This result may be used in combination with the shuffle relations (\ref{eq:defgr}) to prove (see e.g.\ \cite{fm}, Theorem 2) that,
for $\gamma\in\G$,
$$
D_\gamma (gh) = D_\gamma(g)\, D_\gamma(h).
$$
By considering the coordinate mappings $g(x) = x^j$, $h(x)=x^\ell$ and (\ref{eq:wd}) we conclude that
$$
x^i(W_\gamma(x))\, x^i(W_\gamma(x)) = D_\gamma(x^i) D_\gamma(x^\ell)= D_\gamma(x^ix^\ell)
$$
and then linearity shows that, for each polynomial mapping $P$, $P(W_\gamma(x)) = D_\gamma P(x)$.
It follows that
\begin{equation}\label{eq:P}
g(W_\gamma(x)) = D_\gamma g(x),\qquad \gamma\in\G;
\end{equation}
for any (scalar or vector valued) smooth mapping $g$. Thus $D_\gamma g(x)$ provides the formal expansion of the composition $g(W_\gamma(x))$  provided that the coefficients $\gamma$ belong to the group $\G$.

The proof of the formula (\ref{eq:act}), that  defines an action of the group $\G$ on the vector space of all word series,
 is now easy:
$$
W_\delta(W_\gamma(x)) = D_\gamma W_\delta(x) = (D_\gamma\cdot D_\delta)\, x = D_{\gamma\star\delta}\, x = W_{\gamma\star\delta}(x);
$$
we have successively used (\ref{eq:P}), (\ref{eq:wd}), (\ref{eq:compD}) and once more (\ref{eq:wd}).

In (\ref{eq:calculote}), the expression $(\partial_x W_\delta(x))W_\beta$ is the result of applying to the word series $W_\delta(x)$,  the
first-order differential operator associated with the formal vector field $W_\beta(x)$. The formula then reveals that the action
of the algebra $\g$ on word series corresponds to the operation $\beta\star\delta$.

\subsubsection{Linear differential equations}

In Section~\ref{sec:words} it was proved that the initial value problem (\ref{eq:odebeta}) has a unique solution with $\alpha(t)\in\C^\W$ for each $t$. We show here that in fact $\alpha(t)\in\G$. We use the following auxiliary result:

\begin{lemma} Assume that $\eta \in \C^{\W}$ is such that, for some positive integer $n$, and for each $w\in\W_l$, $w^\prime\in\W_m$, $l+m\leq n$, the shuffle relation (\ref{eq:defgr}) hold. Then, for  $\beta\in\g$, $w\in\W_l$, $w^\prime\in\W_m$, $l+m\leq n$, with $w\sh w^\prime= \sum_j w_j$:
$$
\sum_j (\eta \star \beta)_{w_j} = \eta_{w} (\eta \star \beta)_{w^\prime} + (\eta \star \beta)_{w} \eta_{w^\prime}.
$$
\end{lemma}
\begin{proof} Since $\beta\in\g$, there exists a curve $\gamma(t)$ in $\G$ such that (\ref{eq:derivada}) holds. The hypothesis of the lemma then allows us to write:
$$
\sum_j (\eta \star \gamma(t))_{w_j} = (\eta \star \gamma(t))_{w}  (\eta \star \gamma(t))_{w^\prime}.
$$
The result is obtained by applying $ \left.{d}/{dt}\right|_{t=0}$ to both sides of this equality.
\end{proof}

Now consider the solution $\alpha(t)\in\C^\W$ of (\ref{eq:odebeta}). We shall prove by induction on $n$ that for each $t$
\begin{equation}
\label{eq:alphacond_n}
 \sum_j \alpha(t)_{w_j} = \alpha(t)_{w}  \alpha(t)_{w^\prime}
\end{equation}
for  $w\in\W_l$, $w^\prime\in\W_m$, $l+m\leq n$, with $w\sh w^\prime= \sum_j w_j$.

This trivially holds for $n=0$, since  $\alpha(t)_{\emptyset}=1$ for all $t$.
Assume that (\ref{eq:alphacond_n}) is satisfied for some $n\geq 0$, and choose  $w\in\W_l$, $w^\prime\in\W_m$, $l+m\leq n+1$, with $w\sh w^\prime= \sum_j w_j$. From (\ref{eq:odebeta}) we find
\begin{eqnarray*}
  &&\frac{d}{dt} \left(\sum_j \alpha(t)_{w_j} - \alpha(t)_{w}
    \alpha(t)_{w^\prime} \right) =  \\ && \qquad \sum_j(\alpha(t) \star
  \beta(t))_{w_j} - (\alpha(t) \star \beta(t))_w
  \alpha(t)_{w^\prime} -\alpha(t)_w (\alpha(t) \star
  \beta(t))_{w^\prime}
\end{eqnarray*}
and the lemma implies that the right-hand side of this equality vanishes.
Since (\ref{eq:alphacond_n}) holds at $t=0$, it does so for each value of $t$.

\subsection{Proof of Theorem~\ref{th:theoremnormal}}
\label{ss:proofnormal}

We simplify the system (\ref{eq:systnormal}) by performing a sequence of changes of variables with coefficients
$\kappa^{([1])}$, $\kappa^{([2])}$,  \dots in $\G$ so that the change defined by $\kappa^{([n])}$ simplifies the coefficients of the vector field associated with words with $n$ letters and leaves unaltered the coefficients associated with shorter words. The element $\kappa^{([n])}$ is sought in the form $\exp_\star(\lambda^{([n])})$ where $\lambda^{([n])}\in\g$ and $\lambda^{([n])}_w =0$ if $w\in\W\setminus\W_n$. If $\beta^{[n-1]}$ and $\beta^{[n]}$ are respectively the vector fields before and after the $n$-th change of variables, and $w=\bk_1\cdots\bk_n$, the equation (\ref{eq:normaleqn}) implies, after taking into account that $\kappa^{([n])}$ vanishes for nonempty words with less than $n$ letters
$$
i((\bk_1+\cdots\bk_n)\cdot \omega) \kappa_w = \beta^{[n-1]}_w-\beta^{[n]}_w.
$$
If $(\bk_1+\cdots\bk_n)\cdot \omega\neq 0$ we may choose $\lambda^{[n]}_w$ to enforce $\beta^{[n]}_w=0$. In other case, we
set $\beta^{[n]}_w = \beta^{[n-1]}_w$ and $\lambda^{[n]}_w=0$. The element $\lambda^{[n]}$ constructed in this way belongs to $\g$ because the required shuffle relations hold (if shuffling two words leads to resonant words all the coefficients in $\lambda^{[n]}$ vanish; in the nonresonant case the coefficients $\lambda^{[n]}$ are proportional to the corresponding coefficients in $\beta^{[n-1]}$, which satisfy the shuffle relations). In turn $\beta^{([n]}\in\g$ because
$$
\beta^{[n]} = \kappa^{[n]}\star\beta^{[n-1]}\star\big(\kappa^{[n]}\big)^{-1}
- ( \xi_\omega \kappa^{[n]})\star\big(\kappa^{[n]}\big)^{-1};
$$
both terms of the right-hand side are in $\g$ (the second is the value at $t=0$ of $$(d/dt) \Big(\big(\Xi_{\omega t}\kappa^{[n]}\big)\star\big(\kappa^{[n]}\big)^{-1}\Big)$$ and, as noted above $\Xi_{\omega t}\kappa^{[n]}\in\G$).

\subsection{Proof of Theorem~\ref{th:estim}}
\label{ss:proofestimate}

 For $(y,\theta) \in\Omega$ the function $f$ in (\ref{eq:ode2}) is bounded and Lipschitz continuous. Therefore, for $|t|$ small
$(y(t),\theta(t)) = \phi_t(y_0,\theta_0)$ is well defined and $|y(t)-y_0| \leq C_1 |t|$, where $C_1$ is a bound for $|f|$. A simple contradiction argument shows that $|y(h)-y_0|< R$ for $|h| < R/C_1$.

To deal now with the numerical solution, define the intermediate points (stages), $j=1,\dots,r$,
$$
(y_j,\theta_j) = \phi^{(P)}_{b_jh}\Big(  \phi^{(U)}_{a_jh}(y_{j-1},\theta_{j-1})   \Big) =
\phi^{(P)}_{b_jh}(y_{j-1}, \theta_{j-1}+a_jh\omega).
$$
If $|b_jh|< R/(C_1/r)$, the iteration of the argument used above ensures  that $|y_j-y_{j-1}|< R/r$, $j=1,\dots,r$ and then the triangle inequality implies that
$\widetilde{\phi}_h(y_0,\theta_0)\in \Omega$.

We shall use the notations $\overline{W}^{(N)}_{(t\omega,\alpha(t))}(x)$,  $\overline{W}^{(N)}_{(h\omega,\widetilde{\alpha}(t))}(x)$, to refer to the result of suppressing all terms  corresponding to words with $N$ or more letters of the extended word series with coefficients $\alpha(t)$, $\widetilde{\alpha}(h)$ respectively (of course the alphabet is now $\I$ rather than $\Z^d$). In addition we set
$$
\calR_h^{(T)}(x) = \phi_h(x)-\overline{W}^{(N)}_{(h\omega,\alpha(h))}(x),\qquad
\calR_h^{(S)}(x)= \widetilde{\phi}_h(x)-\overline{W}^{(N)}_{(h\omega,\widetilde{\alpha}(h))}(x)
$$
(the superscripts $T$ and $S$ mean \lq true\rq\ and \lq splitting\rq). Our task is to bound $\calR_h(x_0) = \calR_h^{(S)}(x_0)-\calR_h^{(T)}(x_0)$. For
$\calR_h^{(T)}(x_0)$, by stopping the iterative procedure (see e.g.\ \cite{orlando}) that leads to (\ref{eq:solucionprev})
we find the following representation:
\begin{eqnarray*}
\calR_h^{(T)}(x_0)&=&\sum_{\bk_1\cdots \bk_N\in\W_N}\int_0^h dt_N\, \exp(i\bk_N\cdot \omega t_N)
\cdots \\ &&\qquad\qquad\qquad\qquad\int_0^{t_2} dt_1\,\exp(i\bk_1\cdot \omega t_1) f_{\bk_1\cdots \bk_N}(\phi_{t_1}(x_0)).
\end{eqnarray*}
 We know that $\phi_{t_1}(x_0)\in \Omega$ for $|h|\leq h_0$, and therefore we may guarantee that $|\calR_h^{(T)}(x_0)|\leq C|h|^N$, with $C$ depending only on $\I$ and bounds for the derivatives of the Fourier coefficients.

 For $\calR_h^{(S)}(x_0)$ we use a similar device. The key point is that (cf.\ (\ref{eq:solucionprev})--(\ref{eq:solucion})) $$\widetilde{\phi}_h(y_0,\theta_0) = (0,h(a_1+\cdots +a_r)\omega) + (y(h),\eta(h)),$$
 where $(y(t),\eta(t))$ is the solution of
 $$
 \frac{d}{dt} \left[ \begin{matrix}y\\ \eta\end{matrix}\right]
= \sum_{\bk \in\Z^d} \widetilde{\lambda}_\bk(t)\:f_\bk(y,\eta),
$$
with piece-wise constant functions defined by
$$
\widetilde{\lambda}_\bk(t) = r b_j \exp\big(i\bk\cdot \omega (a_1+\cdots+a_j)h\big),\qquad (j-1)h/r \leq t < jh/r, \quad 1\leq j\leq r.
$$
This differential system associated with $\widetilde{\phi}_h$ is very similar to the system (\ref{eq:odeeta}) associated with $\phi_h$, the difference being that  for the former the complex exponentials
are frozen at the times $(a_1+\cdots+a_j)h$. After this observation the residual for $\widetilde{\phi}_h$ is bounded with the technique used for the true $\phi_h$.

\bigskip

{\bf Acknowledgement}. A. Murua and J.M.
Sanz-Serna have been supported by proj\-ects MTM2013-46553-C3-2-P and MTM2013-46553-C3-1-P from Ministerio de Eco\-nom\'{\i}a y Comercio, Spain. Additionally A. Murua has been partially supported by the Basque Goverment  (Consolidated Research Group IT649-13).

\section*{Appendix: examples of perturbed integrable problems}
Any system
\begin{equation}\label{eq:M}
\frac{d}{dt} w = M w+F(w),
\end{equation}
where $M$ is a skew-symmetric $D\times D$ constant matrix, may be brought by a linear change of variables to the  form:
\begin{eqnarray}
(d/dt) z^j &=&  \phantom{-\omega_{\ell}^2 Q^{\ell}+{}}f^j(z,P,Q),\qquad 1\leq j\leq D-2d,\label{eq:z}\\
(d/dt) P^\ell & =&  -\omega_{\ell}^2 Q^{\ell}+g^\ell(z,P,Q),\qquad 1\leq \ell\leq d,\nonumber\\
(d/dt) Q^\ell  &= &\phantom{-\omega_{\ell}^2 }P^{\ell}+ h^\ell(z,P,Q),\qquad  1\leq \ell\leq d\nonumber
\end{eqnarray}
Here $d$ is the number of nonzero eigenvalue pairs $\pm i\omega_\ell$, $\omega_\ell>0$, of $M$.
In the unperturbed case, $f^i\equiv 0$, $g^\ell\equiv 0$, $h^\ell\equiv 0$, the system consists of $d$ uncoupled harmonic oscillators with frequencies $\omega_\ell$,  together with $D-2d$ trivial equations $(d/dt) z^j = 0$. The introduction of the variables $a^\ell$, $\theta^\ell$ such that
\begin{equation}\label{eq:aav}
P^\ell = \sqrt{2\omega a^\ell} \cos \theta^\ell, \qquad Q^\ell = \sqrt{\frac{2a^\ell}{\omega_\ell}} \sin \theta^\ell,\qquad 1\leq \ell\leq d,
\end{equation}
takes now the system to the format (\ref{eq:ode2}) with $y = (z^1,\dots, z^{D-2d},a)$.
The system (\ref{eq:M}) or (\ref{eq:z}) is a natural candidate to integration by splitting methods based on separating the linear part (that may be integrated in closed form) from the perturbation. The later may perhaps be treated by means of a numerical integrator with a very fine time step. In favorable instances, the perturbation may be  integrated analytically in closed form; this is the situation in the following particular case of (\ref{eq:z}), commonly found in mechanics ($D$ is even and $z = (p,q)$),
\begin{eqnarray}
 (d/dt) p^j &=& \phantom{ -\omega_\ell^2 Q^\ell+{}}f^j(q,Q),\qquad  1\leq j\leq D/2-d,\label{eq:pqPQ}\\
 (d/dt) q^j &=&\phantom{ -\omega_\ell^2}p^j,\nonumber\\
 (d/dt) P^\ell &=& -\omega_\ell^2 Q^\ell+g^\ell(q,Q),\qquad 1\leq \ell\leq d,\nonumber\\
 (d/dt) Q^\ell &=& \phantom{ -\omega_\ell^2}P^\ell.\nonumber
\end{eqnarray}
Under the dynamics of the perturbation, $p$ and $P$ remain constant and $q$ and $Q$ grow linearly with $t$.

The  system (\ref{eq:pqPQ}) is Hamiltonian if the forces $f^j$, $g^\ell$ derive from a potential. When that happens, the introduction of the canonical
($dP^j\wedge dQ^j = da^j\wedge d\theta^j$) action/angle variables in (\ref{eq:aav}) preserves the Hamiltonian character of the equations of motion and even the value of the Hamiltonian function.

So far the unperturbed problem has been linear, but nonlinear cases may also be treated. Typically, integrable nonlinear problems may be brought to the form (\ref{eq:unpertur}) with $\omega=\omega(y)$; a device commonly used e.g.\ in dynamical astronomy consists in fixing a relevant value $y_0$ of $y$, decomposing
$\omega(y) = \omega(y_0)+\Delta(y)$ and seeing $(0,\Delta(y))$ as part of the perturbation.

 There are many instances of perturbations of nonlinear integrable systems, after the introduction of suitable action/angle variables take the form (\ref{eq:ode2}). A well-known example is provided by perturbations of the Keplerian motion of a celestial body.

\begin{remark} \em For (\ref{eq:pqPQ}),  as we just noticed,  the system  corresponding to the perturbation may be solved in closed form in the variables
$(p,q,P,Q)$. On the other hand, the analysis in Section~\ref{sec:split} operated with a {\em different} set of variables $x = (y,\theta) =(p,q,a,\theta)$. This causes no difficulty: it is standard practice when using splitting methods that the different split systems are integrated employing  different sets of dependent variables. In  partial differential equations, parts corresponding to linear, constant-coefficient differential operators are typically integrated in Fourier space and nonlinearities in physical space. Splitting methods are based on true solution flows, which of course {\em commute with changes of variables}. The situation is very different for, say, Runge-Kutta schemes, where (except for affine changes) changing variables does not commute with the application of the numerical method, and the performance of the integrator very much depends on the choice of dependent variables. For instance, any consistent Runge-Kutta method integrates exactly the unperturbed problem when written  as in (\ref{eq:unpertur}) but incurs in errors when dealing with the unperturbed version of (\ref{eq:z})  ($f^i\equiv 0$, $g^\ell\equiv 0$, $h^\ell\equiv 0$).
\end{remark}

\end{document}